\documentclass[review]{elsarticle}

\usepackage{lineno,hyperref}
\modulolinenumbers[5]
\journal{Advances in Applied Mathematics and Mechanics}

\usepackage{amsmath,amssymb,amsthm,amsfonts,afterpage,float,bm,stmaryrd,paralist,wrapfig,bbm,soul}

\usepackage{cancel}
\usepackage{color}
\usepackage{fancyhdr}
\usepackage{graphics}
\usepackage{hyperref}
\usepackage{graphicx,epsf}
\usepackage{makeidx}
\usepackage{epsfig}
\usepackage{lscape}
\usepackage{subfigure}
\usepackage{listings}
\usepackage{multirow}

\usepackage{empheq,xcolor}
\usepackage{amsmath,eqparbox,xparse,lipsum}

\pssilent

\newtheorem{theorem}{Theorem}[section]
\newtheorem{lemma}{Lemma}[section]

\newtheorem{definition}{Definition}[section]

\numberwithin{equation}{section}
\numberwithin{figure}{section}
\numberwithin{table}{section}

\def\XXint#1#2#3{{\setbox0=\hbox{$#1{#2#3}{\int}$}
\vcenter{\hbox{$#2#3$}}\kern-.51\wd0}}

\graphicspath{{Figure/}}

\begin{document}

\setlength{\pdfpageheight}{\paperheight}
\setlength{\pdfpagewidth}{\paperwidth}
\title{Ultraspherical Spectral Method for Block Copolymer Systems on Unit Disk}
\author{Wangbo Luo, Yanxiang Zhao}
%\author{Hyunjung Choi}
\address{Department of Mathematics, George Washington University, Washington D.C., 20052}
%\author{Yanxiang Zhao}
%\address{Department of Mathematics, George Washington University, Washington D.C., 20052}
\fntext[myfootnote]{Corresponding author: yxzhao@email.gwu.edu}

\begin{abstract}

In this paper, we investigate ultraspherical spectral method for the Ohta-Kawasaki (OK) and Nakazawa-Ohta (NO) models in the disk domain, representing diblock and triblock copolymer systems, respectively. We employ ultraspherical spectral discretization for spatial variables in the disk domain and apply the second-order backward differentiation formula (BDF) method for temporal discretization. To our best knowledge, this is the first study to develop a numerical method for diblock and triblock copolymer systems with long-range interactions in disk domains. We show the energy stability of the numerical method in both semi-discrete and fully-discrete discretizations. In our numerical experiments, we verify the second-order temporal convergence rate and the energy stability of the proposed methods. Our numerical results show that the coarsening dynamics in diblock copolymers lead to bubble assemblies both inside and on the boundary of the disk. Additionally, in the triblock copolymer system, we observe several novel pattern formations, including single and double bubble assemblies in the unit disk. These findings are detailed through extensive numerical experiments.
\end{abstract}


\date{\today}
\maketitle

\section{Introduction}\label{Intro}
Block copolymers are macromolecules that contain monomers comprising different repeating units covalently linked together in polymer chains. Due to their remarkable ability for self-assembly, block copolymers form various nanoscale-ordered structures at thermodynamic equilibrium \cite{Bats_Fredrickson1990,Bats_Fredrickson1999,Hamley2004,Botiz_Darling2010} and have attracted considerable theoretical and experimental interest over the past several decades \cite{Helfand_Wasserman1976,Leibler1980,Meier1969,Ohta_Kawasaki1986}. Diblock copolymers, referred to as the binary system, comprise two distinct subchains made up of species $A$, and $B$, respectively. In contrast, triblock copolymers, known as the ternary system, consist of three subchains, each containing species $A$, $B$, and $C$, respectively.    

In this paper, we begin with the Ohta-Kawasaki (OK) and Nakazawa-Ohta (NO) models introduced in \cite{Ohta_Kawasaki1986,Nakazawa_Ohta1993}, which are general block copolymer systems with long-range interactions and describe the binary and ternary systems, respectively. In our study, we aim to explore the coarsening dynamic and pattern formations at equilibrium state in the OK/NO models in the disk domain.

Using the phase field labeling function $u = u(x)$ to represent the density of the species $A$ and tracing the concentration of the species $B$ by $1-u(x)$, the OK model introduces the diblock copolymers by the following energy functional \cite{Xu_Zhao2019,Xu_Zhao2020}:
\begin{align}\label{eqn:OK}
        E^{\text{OK}}[u] = \int_{\Omega}\left[\frac{\epsilon}{2}|\nabla u|^2+\frac{1}{\epsilon}W(u)\right] \text{d}x + \frac{\gamma}{2}\int_{\Omega}|(-\Delta)^{-\frac{1}{2}}(u-\omega)|^2 \ \text{d}x,
\end{align}
with a volume constraint
\begin{align*}
    \int_{\Omega}u\ \text{d}x = \omega|\Omega|,
\end{align*}
where $0<\epsilon\ll 1$ is an interface parameter, $\Omega\in \mathbb{R}^2$ is the spatial domain. $W(u) = 18(u^2-u)^2$ is a double well potential which enforces the labeling function $u$ to be $0$ and $1$ except the interfacial region between species $A$ and $B$. The parameter $\gamma$ represents the strength of the long-range repulsive interaction that favors the smaller size of the species and forces them to split in the OK model. Similarly, the free energy functional of the NO model \cite{Nakazawa_Ohta1993} for the triblock copolymer system is defined as follows:
\begin{align}\label{eqn:NO}
    E^{\text{NO}}[u_1,u_2] = & \int_{\Omega}\frac{\epsilon}{2}\left(|\nabla u_1|^2+|\nabla u_2|^2+\nabla u_1 \cdot \nabla u_2\right) \text{d}x + \int_{\Omega} \frac{1}{\epsilon}W_2(u_1,u_2)\ \text{d}x \nonumber\\
    & + \sum_{i,j=1}^{2}\frac{\gamma_{ij}}{2}\int_{\Omega}\left[(-\Delta)^{-\frac{1}{2}}(u_i-\omega_i)\times (-\Delta)^{-\frac{1}{2}}(u_j-\omega_j)\right] \text{d}x,
\end{align}
where $u_i = u_i(x)$, $i=1,2$ are phase field labeling functions representing the density of species $A$ and $B$, respectively. The concentration of species $C$ can be implicitly described by $1-u_1(x)-u_2(x)$. $W_2(u_1,u_2)$ is defined as $W_2(u_1,u_2): = \frac{1}{2}\left[W(u_1)+W(u_2)+W(1-u_1-u_2)\right]$, and the volume constraints for $u_1$ and $u_2$ are
\begin{align*}
        \int_{\Omega}u_i\ \text{d}x = \omega_i|\Omega|, \ \ i = 1,2.
\end{align*}

To implement the volume constraint in energy minimization, we introduce penalty terms to the energy functionals (\ref{eqn:OK} and \ref{eqn:NO}), leading to the unconstrained free energy functionals of the penalized OK model  \cite{Xu_Zhao2019},
\begin{align}\label{eqn:pOK}
    E^{\text{pOK}}[u] & = \int_{\Omega}\left[\frac{\epsilon}{2}|\nabla u|^2+\frac{1}{\epsilon}W(u)\right] \text{d}x + \frac{\gamma}{2}\int_{\Omega}|(-\Delta)^{-\frac{1}{2}}(u-\omega)|^2 \ \text{d}x \nonumber\\
    & + \frac{M}{2}\left(\int_{\Omega}u\ \text{d}x - \omega|\Omega|\right)^2, 
    \end{align}
    and the penalized NO model \cite{Choi_Zhao2021},
    \begin{align}\label{eqn:pNO}
    E^{\text{pNO}}[u_1,u_2] = & \int_{\Omega}\frac{\epsilon}{2}\left(|\nabla u_1|^2+|\nabla u_2|^2+\nabla u_1 \cdot \nabla u_2\right) \text{d}x + \int_{\Omega} \frac{1}{\epsilon}W_2(u_1,u_2)\ \text{d}x \nonumber\\
    & + \sum_{i,j=1}^{2}\frac{\gamma_{ij}}{2}\int_{\Omega}\left[(-\Delta)^{-\frac{1}{2}}(u_i-\omega_i)\times (-\Delta)^{-\frac{1}{2}}(u_j-\omega_j)\right] \text{d}x \nonumber\\
    & + \sum_{i=1}^{2} \frac{M_i}{2}\left(\int_{\Omega}u_i\ \text{d}x - \omega_i|\Omega|\right)^2,
\end{align}
with the penalty constants $M,M_1,M_2\gg 1$. Then by considering the corresponding penalized $L^2$ gradient flow dynamics for the energy functionals (\ref{eqn:pACOK}, \ref{eqn:pACNO}), we derive the penalized Allen-Cahn-Ohta-Kawasaki (pACOK) equation for the time evolution $u(x,t)$ with given initial data $u(x,t=0) = u_0(x)$, 
\begin{align}\label{eqn:pACOK}
    \frac{\partial u}{\partial t} = -\frac{\delta E^{\text{pOK}}[u]}{\delta u} = \epsilon\Delta u - \frac{1}{\epsilon}W'(u) - \gamma(-\Delta)^{-1}(u - \omega) - M \left(\int_{\Omega}u\ \text{d}x - \omega|\Omega|\right), 
\end{align}
and the penalized Allen-Cahn-Nakazawa-Ohta (pACNO) equations for the time evolution $u_i(x,t)$, $i=1,2$ with given initial data $u_i(x,t=0) = u_{i,0}(x)$,
\begin{align}\label{eqn:pACNO}
        \frac{\partial u_i}{\partial t}  = &\ -\frac{\delta E^{\text{pNO}}[u_1,u_2]}{\delta u_i} \nonumber \\
        = &\  \epsilon\Delta u_i + \frac{\epsilon}{2}\Delta u_j- \frac{1}{\epsilon}\frac{\partial W_2}{\partial u_i} - \gamma_{ii}(-\Delta)^{-1}(u_i- \omega)\nonumber \\
        & - \gamma_{ij}(-\Delta)^{-1}(u_j - \omega)
         - M_i \left(\int_{\Omega}u_i\ \text{d}x - \omega_i|\Omega|\right), \ \ i,j = 1,2. 
\end{align}

\subsection{Previous work and our contributions}

There have been extensive theoretical and numerical studies on the OK and NO models in recent years. In \cite{Nishiura_Ohnishi1995,Ren_Wei2000}, the authors presented a simpler analogy of a binary inhibitory system derived from the OK model for diblock copolymers. Meanwhile, Ren and Wei \cite{Ren_Wei2003_1,Ren_Wei2003_2} derived the ternary system from the NO model for triblock copolymers. Additionally, Wang, Ren, and Zhao \cite{Wang_Ren_Zhao2019} numerically and theoretically studied the bubble assemblies in the NO model while varying the system parameters. To study the equilibria and dynamics of the binary and ternary systems numerically, Xu and Zhao \cite{Xu_Zhao2019,Xu_Zhao2020} presented the $L^2$ gradient flow for the energy functional of the OK and NO models and introduced the corresponding  ACOK and ACNO equations. Similar to the energy stable methods for various gradient flow dynamics of different equations \cite{Chen_Conde2012,Hu_Wise2009,Shen_Yang2010,Shen_Wang2012,Wang_Wise2011,Wise_Wang2009}, Luo, Zhao and Choi \cite{Choi_Zhao2021,LuoZhao_Arxiv2023} proposed the stabilized second-order semi-implicit methods in time-discretization and used the Fourier spectral method in spatial domain to solve the ACOK and ACNO equations in the square domain with periodic boundary conditions. Furthermore, Luo and Zhao \cite{LuoZhao_PhysicaD2024,LuoZhao_Arxiv2023} introduced a generalized OK model and perform several numerical and theoretical studies for their model with novel square lattice patterns.

In the binary system, Ren and Shoup \cite{Ren_Shoup2020} investigated the stationary bubble assemblies of the OK model in the disk domain with the Neumann boundary condition and studied the energy difference between the number of interior discs and boundary half discs. The other theoretical works by Wang and Ren \cite{Wang_Ren2017,Wang_Ren2019} explored the explicit form of Green's function of the Laplacian operator in the disk domain. This explicit form played a key role in studying the bubble's interaction with the disk boundary (boundary effect) in the ternary system. However, it is difficult to predict the pattern formation at the equilibrium state for both the binary and ternary systems from analytical studies. Therefore, numerical studies are necessary to guide us to understand the equilibria and coarsening dynamics of the OK and NO models in the disk domain. Furthermore, numerical findings may suggest the new directions for the future theoretical research and practical application.

In contrast to the square domain, where periodic boundary conditions can be efficiently handled using the Fourier spectral method, one of the most challenging problems we faced in this project is to solve the fully discrete systems in the disk domain with the Neumann boundary condition.  A highly effective approach for handling problems in polar geometries is to map the disk domain to a rectangular domain using polar coordinates. By employing polar coordinates, the angular direction becomes periodic, allowing for efficient treatment using techniques like the fast Fourier transform (FFT). Therefore, by carefully addressing the boundary conditions, spectral methods are remarkably well-suited for solving problems in polar geometries. In the 1970s, several spectral-collocation or spectral-tau methods were developed to compute functions in polar geometries \cite{CANUTO_HUSSAINI_QUARTERONI_ZANG1987,
Fornberg1995,
GOTTLIEBANDS_ORSZAG1977,
HUANGANDD_SLOAN1993,
EISEN_HEINRICHS_WITSCH1991}. Eisen et al.(1991) \cite{EISEN_HEINRICHS_WITSCH1991} proposed an algorithm based on an odd-even parity argument for the Cartesian and polar coordinates \cite{Boyd2001,Boyd_Yu2011} and the tau method \cite{Ortiz1969}, which had quasi-optimal computational complexity but was limited to second-order equations. Subsequently, Shen \cite{Shen1995,Shen2003,Shen_Wang2007} developed efficient Spectral-Galerkin methods using bases for second, third, fourth, and higher odd-order differential equations with constant coefficients in polar geometries. However, these methods faced challenges in automatically imposing boundary conditions. In 2011, Boyd and Yu \cite{Boyd_Yu2011} provided a comprehensive overview of numerical methods for computing functions on unit disk and solving Poisson's equation in polar geometries associated with a broad range of strategies.

Recently, Olver and Townsend \cite{Olver_Townsend2013} introduced the so called ultraspherical spectral method for solving differential equations with variable coefficients and general boundary conditions. The ultraspherical spectral method utilizes Chebyshev and ultraspherical polynomial bases, depending only on the order of the differential equation, rather than the specific boundary conditions. This approach leads to nearly banded matrices, preserving efficiency and accuracy in computations. Furthermore, the ultraspherical spectral method is not limited to polar geometries. In fact, Olver and Townsend \cite{Olver_Townsend2015} demonstrated its effectiveness in solving linear partial differential equations (PDEs) with general boundary conditions on rectangular domains. This flexibility makes the ultraspherical spectral method a valuable tool for a wide range of problems in different domains. Employing the ultraspherical spectral method, Wilber, Townsend, and Wright \cite{Wilber_Townsend_Wright2017} developed algorithms on computing functions such as integration, function evaluation, vector calculus on unit disk, in addition to devising a fast disk Poisson solver in polar geometries. Their algorithms applied a structure-preserving variant of iterative Gaussian elimination which enabled the adaptive construction of low-rank approximations for functions in polar geometries. This approach significantly reduced the computational cost associated with various operations while maintaining the ability for fast and spectrally accurate computations in the disk domain.

The main contribution of our work is two-folded. Firstly, we apply the stabilized second-order linear semi-implicit scheme \cite{Choi_Zhao2021,LuoZhao_Arxiv2023} for the pACOK and pACNO equations, where all nonlinear and nonlocal terms are treated explicitly. Incorporating the Neumann boundary conditions, we then solve the fully discrete systems using the ultraspherical spectral method \cite{Olver_Townsend2013} with low-rank approximations \cite{Wilber_Townsend_Wright2017} to efficiently solve the phase field functions in polar geometries at each time step. To the best of our knowledge, this is the first numerical method developed to study binary and ternary systems with long-range interactions over disk domains. Secondly, we investigate pattern formation at the equilibrium state for the OK and NO models in the disk domain. Our observations reveal various types of bubble assemblies, some of which have not been previously reported and deserved further numerical and theoretical exploration. The interior double bubble assemblies in the ternary system, as illustrated in Section \ref{subsec:NO}, are particularly noteworthy. 

\subsection{Notations}

In this paper, for the discussion of energy stability in the rest of the paper, we modify the function $W(s)$ with a quadratic approach. This modification ensures that $W''$ possess a finite upper bound, a prerequisite for the energy stable schemes in Ginzburg–Landau type dynamics \cite{Shen_Yang2010}. We adopt the quadratic extension $\Tilde{W}(u)$ of $W(u)$ as used in \cite{Shen_Yang2010} and referenced in related citations. For simplicity in our discussion, we continue to use $W(u)$ to represent $\Tilde{W}(u)$, and denote $ L_{W''}:= \|W''\|_{L^{\infty}}$, namely, $L_{W''}$ denote the upper bound for $|W''|$.

We adopt the similar notations from \cite{Xu_Zhao2019}. Let the disk domain $\Omega = [0,1]\times[0,2\pi)\subset\mathbb{R}^2$, and the space $\mathring{H}^{s}(\Omega)$ is given by
\begin{align*}
   & \mathring{H}^{s}(\Omega) = \left\{u\in H^s(\Omega): \ u \text{ is homogeneous Neumann in } r, \text{periodic in }\theta, \text{ and } \int_{\Omega}u(x)\ \text{d}x =0 \right\}, 
\end{align*}
consisting of all functions $u\in H^s(\Omega)$ which satisfies the Neumann boundary condition in $r$-direction and the periodic boundary condition in $\theta$-direction with zero mean. We use $\|\cdot\|_{H^s}$ to represent the standard Sobolev norm and $L^2(\Omega) = H^0(\Omega)$.

The inverse Laplacian operator defined on the unit disk $(-\Delta)^{-1}: \ \mathring{L^{2}}(\Omega) \to \mathring{H}^{1}(\Omega)$ follows:
\begin{align}\label{eqn:LapInv}
   & (-\Delta)^{-1}v = u \Longleftrightarrow  -\Delta u = -u_{rr}-\frac{1}{r}u_{r}-\frac{1}{r^2}u_{\theta\theta} = v, \notag\\
   &  u_{r}(r=1,\theta) = 0, \ \ u \text{ is periodic in }\theta, \text{ and homogeneous Neumann in } r
\end{align}
We denote $\|(-\Delta)^{-1}\|_{L^2}$ as the optimal constant such that $\|(-\Delta)^{-1}f\|_{L^{2}}\leq C\|f\|_{2}$ \cite{Canuto_Springer2006}. Note that the constant $C$ is bounded and depends on $\Omega$ which is a simple consequence of the elliptic regularity for general smooth domains.

The rest of the paper is organized as follows. In section \ref{sec2}, we will review the basic ideas of the ultraspherical spectral method, including its conversion, differentiation, multiplication, and application to the Poisson equation with the Dirichlet and Neumann boundary conditions. In section \ref{sec3}, we will present the second-order BDF energy stable schemes with semi- and fully-discrete systems for the pACOK and pACNO equations in the disk domain and prove their energy stability. In section \ref{sec:numerical}, we will present some numerical experiments for the second-order convergence rate, the coarsening dynamics, and the equilibrium state of the pACOK and pACNO equations in the disk domain. These numerical findings will display single-bubble assembly for the OK model, and several types of novel pattern formation with both single and double bubble assemblies on the interior of the disk domain for the NO model.

\section{Preliminary Review}\label{sec2}

In the past decades, two main challenges are faced when dealing with functions in polar geometries: maintaining regularity at the original of the disk, and using fast transforms for numerical computation. Most methods \cite{Fornberg1995,Shen2000,EISEN_HEINRICHS_WITSCH1991} address only one of the challenges by representing functions with Chebyshev-Fourier expansion. Recently, \cite{Wilber_Townsend_Wright2017} proposed a novel approach using Chebyshev-Fourier expansion, providing three key benefits: a structure that allows for fast transforms, regularity at the origin of the disk, and near-optimal sampling. This method combines low-rank approximations for functions in polar geometries with an adaptive interpolation method that uses an iterative Gaussian elimination process and a strategy similar to the double Fourier sphere (DFS) method \cite{Fornberg1995}.

To approximate smooth functions $f(r,\theta)$ in polar geometries, low-rank approximation is adopted as a sum of a relatively small number $K$ of rank-1 functions to machine precision. To maintain regularity and symmetry of the function, \cite{Wilber_Townsend_Wright2017} defined the Type-II Block-Mirror-Centrosymmetric  (BMC-II) functions (\ref{def:BMC-II}) to represent the smooth functions on the unit disk, which guarantees that the function remains smooth and continuous at the origin and preserves its geometric properties. 
\begin{definition}[BMC-II function \cite{Wilber_Townsend_Wright2017}]\label{def:BMC-II}
A function $\tilde{f}:[-\pi,\pi]\times[-1,1]\to\mathbb{C}$ is a Type-II Block-Mirror-Centrosymmetric (BMC-II) function if there are functions $g,h:[0,\pi]\times[0,1]\to\mathbb{C}$ such that 
\begin{align}\label{BMC-structure}
\tilde{f} = \begin{bmatrix}
g & h \\
\mathrm{filp}(h) & \mathrm{flip}(g)
\end{bmatrix}
\end{align}
where $\mathrm{filp}$ is a glide reflection in group theory, and $f(\cdot,0) = \alpha$, where $\alpha$ is a constant. Note that (\ref{BMC-structure}) is called a BMC structure. A BMC function with $f(\cdot,0)$ equals a constant is BMC-II.
\end{definition}
To construct low-rank approximations of BMC-II functions while preserving the BMC-II structure, a parity-based interpretation of structure-preserving iterative Gaussian elimination procedure (GE) has been developed \cite{Wilber_Townsend_Wright2017}, which ensures that the pivot selection and each rank-1 update maintain the symmetry and regularity of the function. By using GE, the low-rank approximation of $\tilde{f}(r,\theta)$ is constructed by 
\begin{align*}
\tilde{f}(r,\theta) \approx \sum_{j=1}^{K}a_jc_j(r)d_j(\theta)
\end{align*}
where $a_j$ is a coefficient related to the GE pivots, $c_j(r)$ and $d_j(\theta)$ are the $j$th column slice and row slice, respectively, constructed during the GE procedure, and $K$ is a relatively small rank of the approximation \cite{TownsendTrefethen_SJSC2013}. 

The low-rank approximation method samples functions over the unit disk in a way that avoids oversampling near the origin, ensuring a near-optimal interpolation grid, and is particularly effective for operations such as point-wise evaluation, integration, and differentiation, simplifying them to essentially one-dimensional procedures.  Using this idea, Townsend \cite{Trefethen_ChebGuide2014} has built the computational frameworks for differentiation, integration, function evaluation, and vector calculus in polar geometries, which is publicly available through the open source Chebfun package written in MATLAB. In our work, we directly use this package to work with functions in polar geometries. This includes evaluating polar functions on mesh grids, converting polar functions from and to Chebyshev-Fourier basis, and computing integrals of polar functions.  

To efficiently solve the linear differential equations in Chebyshev-Fourier expansion with variable coefficients and general boundary conditions, Olver and Townsend \cite{Olver_Townsend2013} developed the ultraspherical spectral method, which results in the almost banded matrices to an almost banded, well-conditioned linear system. In the rest part of this section, we review this method, focusing on the matrix representations for the differentiation, multiplication, and conversion operators. We then discuss its application in the spectral collocation method \cite{ShenTangWang_Springer2011} for solving the boundary value problem with Dirichlet and Neumann boundary condition, and its extention to the 2D Helmholtz equation in the disk domain \cite{Wilber_Townsend_Wright2017}.

\subsection{Matrix operators for the ultrasherical spectral mehod}\label{subsec:matrixoperator}

The ultraspherical spectral method employs the differentiation rule \cite{Olver_Townsend2013} to create the recurrence relations of the $m$-th derivative of Chebyshev polynomials
\begin{align}\label{rel_ult_chebT}
    \frac{\text{d}^m}{\text{d}x^m}T_{n+m}(x) = 2^{m-1}(m-1)!(n+m)C_n^{(m)}(x), \ \ \forall m\in\mathbb{N}, \ n\geq0.
\end{align}
where $C_n^{(m)}(x)$ is the degree-$n$ ultraspherical polynomials of parameter $m$ \cite{Cohl_MacKenzie_Volkmer2013}, and these recurrence relations result in a sparse representation of differentiation operators for the $m$-th derivative of Chebyshev polynomials and ultraspherical polynomials. Denote $\mathcal{D}_{0,m}$ as the differentiation operator which maps a vector of Chebyshev ${T_n}$ coefficients to a vector of ultraspherical $C_n^{(m)}$ coefficients of the $m$-th derivative. The sparse representation matrix  is given by
\begin{align}\label{mat_diff_chebT_ult}
\mathcal{D}_{0,m} = 2^{m-1}(m-1)!
   \begin{pmatrix}
   \overbrace{
   \begin{array}{ccc}
       {0} & {\cdots} & {0} 
   \end{array}
   }^{m\ \text{times}}
          & m &  0       \\
             & 0  & m+1 &  0  \\
             &   &  0   & m+2 & 0 \\
             &   &     &  \ddots   & \ddots & \ddots  
   \end{pmatrix}, \ \ m \geq1.
\end{align}

Since the ultraspherical spectral method switches the coefficients between bases, it's important to find the operators that convert between the Chebyshev and ultraspherical bases. Under the recurrence relations in \cite{Cohl_MacKenzie_Volkmer2013},
\begin{align*}
    T_k =
    \left\{
    \begin{array}{lr}
       \frac{1}{2}\left(C_k^{(1)}-C_{k-1}^{(1)}\right),  & k\geq2, \\
        \frac{1}{2}C_1^{(1)}, & k=1, \\
        C_0^{(1)}, & k=0,
    \end{array}
    \right. \ \ \ \text{and} \ \ \ 
    C_k^{(m)} =
    \left\{
    \begin{array}{lr}
       \frac{m}{m+k}\left(C_k^{(m+1)}-C_{k-2}^{(m+1)}\right),  & k\geq2, \\
        \frac{m}{m+1}C_1^{(m+1)}, & k=1, \\
        C_0^{(m+1)}, & k=0,
    \end{array}
    \right.
\end{align*}
 $\mathcal{S}_{0,1}$  denote the conversion operator between the Chebyshev $T_k$ basis and ultraspherical $C^{(1)}$ basis, and $\mathcal{S}_{n,n+1}$ for $n\geq 1$ converts the ultraspherical $C^{(n)}$ and $C^{(n+1)}$ bases
\begin{align}\label{cov_01}
    \mathcal{S}_{0,1} = 
    \begin{pmatrix}
        1 & 0 & -\frac{1}{2} &       \\
          & \frac{1}{2} & 0 & -\frac{1}{2}    \\
          &   &  \frac{1}{2} & 0 & -\frac{1}{2}  \\
          &   &   &   \ddots & \ddots  & \ddots 
    \end{pmatrix}, \ \ \ 
        \mathcal{S}_{n,n+1} = 
    \begin{pmatrix}
        1 & 0 & -\frac{n}{n+2} &       \\
          & \frac{n}{n+1} & 0 & -\frac{n}{n+3}    \\
          &   &  \frac{n}{n+2} & 0 & -\frac{n}{n+4}  \\
          &   &   &   \ddots & \ddots  & \ddots 
    \end{pmatrix}, \ \ n \geq1.
\end{align}
Therefore, the conversion between the Chebyshev $T_k$ and ultraspherical $C^{(n)}$ basis is denoted by $\hat{\bm{u}}\to \mathcal{S}_{0,n}\hat{\bm{u}} = \mathcal{S}_{n-1,n}\mathcal{S}_{n-2,n-1}\cdots\mathcal{S}_{0,1}\hat{\bm{u}}$ where $\hat{\bm{u}} = [\hat{u}_1,\hat{u}_2,\hat{u}_3,\cdots]^T$ are the Chebyshev $T_k$ coefficients for the function $u(x)$.

To handle the multiplication of two functions of the form $a(x)u(x)$, we expand $a(x)$ and $u(x)$ as the Chebyshev series
\begin{align*}
    a(x) = \sum_{k=0}^{\infty}\hat{a}_kT_k(x), \ \ \ u(x) = \sum_{k=0}^{\infty}\hat{u}_kT_k(x).
\end{align*}
The product $a(x)u(x)$ involves multiplication of basis functions $T_j$ and $T_k$ with $j,k\geq0$, which leads to the Chebyshev coefficients of $a(x)u(x)$ as
\begin{align*}
    a(x)u(x) = \sum_{j=0}^{\infty}\sum_{k=0}^{\infty}\hat{a}_j\hat{u}_kT_j(x)T_k(x) = \sum_{k=0}^{\infty} \hat{b}_kT_k(x).
\end{align*}
Then, the new Chebyshev coefficients vector $\hat{\bm{b}} = [ \hat{b}_0,\hat{b}_1,\hat{b}_2,\cdots]^T$ can be expressed as the multiplication of a matrix $\mathcal{M}_T[a]$ and a coefficient vector $\hat{\bm{u}}$, i.e. $\hat{\bm{b}} = \mathcal{M}_T[a]\hat{\bm{u}}$ where the matrix operator $\mathcal{M}_T[a]$ is of the following form
\begin{align}\label{mat_mult_chebT}
    \mathcal{M}_T[a] = 
    \frac{1}{2}\begin{bmatrix}
    2\hat{a}_0 & \hat{a}_1  & \hat{a}_2  & \hat{a}_3  & \cdots \\
    \hat{a}_1  & 2\hat{a}_0 & \hat{a}_1  & \hat{a}_2  & \ddots \\
    \hat{a}_2  & \hat{a}_1  & 2\hat{a}_0 & \hat{a}_1  & \ddots \\
    \hat{a}_3  & \hat{a}_2  & \hat{a}_1  & 2\hat{a}_0 & \ddots \\
    \vdots     & \ddots     & \ddots     & \ddots     & \ddots
    \end{bmatrix}
    + \frac{1}{2} \begin{bmatrix}
    0 & 0  & 0  & 0  & \cdots \\
    \hat{a}_1  & \hat{a}_2 & \hat{a}_3  & \hat{a}_4  & \cdots \\
    \hat{a}_2  & \hat{a}_3 & \hat{a}_4  & \hat{a}_5  & \reflectbox{$\ddots$} \\
    \hat{a}_3  & \hat{a}_4 & \hat{a}_5  & \hat{a}_6  & \reflectbox{$\ddots$} \\
    \vdots     & \reflectbox{$\ddots$}     & \reflectbox{$\ddots$}     & \reflectbox{$\ddots$}     & \reflectbox{$\ddots$}
    \end{bmatrix}
\end{align}
Moreover, while requiring the ultraspherical expansion of $a(x)u(x)$, we can combine the Chebyshev multiplication operator and the conversion operator (\ref{cov_01}) between Chebyshev and ultraspherical bases to explore the ultraspherical coefficients of $a(x)u(x)$. For example, the $C^{(m)}$ expansion coefficient of $a(x)u(x)$ can be expressed as 
\begin{align}\label{coe_mult}
    \hat{\bm{c}} = \mathcal{S}_{m-1,m}\mathcal{S}_{m-2,m-1}\cdots\mathcal{S}_{0,1}[\mathcal{M}_T[a]\hat{\bm{u}}],
\end{align}
where $\hat{\bm{c}}$ is the vector of coefficients in ultraspherical $C^{(m)}$ basis.

\subsection{Boundary value problem with Dirichlet and Neumann boundary conditions}

In this subsection, we review the ultraspherical spectral collocation method applied to the boundary value problem:
\begin{align}\label{eqn:bvp}
    & u''(x) + u'(x) + a(x)u(x) = f(x), \ \ \ x \in[-1,1], \nonumber \\
    & \text{Dirichlet: } u(1) = u(-1) = 0, \text{or }  \\
    &  \text{Neumann: } u'(1) = u'(-1) = 0. \nonumber
\end{align}

Given a set of Gauss-Lobatto quadrature nodes and weights $\{x_j,\ \omega_j: = \omega(x_j)\}_{k=0}^N$, the solution $u\in C([a,b])$ of the boundary value problem is approximated by 
\begin{align}\label{coe_chebT}
    u(x) = \sum_{k=0}^{N}\hat{u}_k\phi_k(x), \ x \in [-1,1],
\end{align}
where the basis functions $\phi_k(x)$ are the Chebyshev or ultraspherical polynomials of degree-$k$, and it's corresponding coefficients $\hat{u}_k$ are determined by the Gauss-Lobatto points $\{x_k\}_{k=0}^N$, as detailed in \cite{ShenTangWang_Springer2011}.

Take the Chebyshev expansions of the function $u(x)$, $a(x)$, and $f(x)$:
\begin{align*}
    u(x) = \sum_{k=0}^{N}\hat{u}_kT_k(x),\ \ \ a(x) = \sum_{k=0}^{N}\hat{a}_kT_k(x),  \ \ \ f(x) = \sum_{k=0}^{N}\hat{f}_kT_k(x).
\end{align*}
Define $\hat{\bm{u}} = [\hat{u}_1,\hat{u}_2,\cdots,\hat{u}_N]^{T}$ and $\hat{\bm{f}} = [\hat{f}_1,\hat{f}_2,\cdots,\hat{f}_N]^{T}$ as the vector of coefficients in Chebyshev basis. In terms of differentiation operator in (\ref{mat_diff_chebT_ult}), we represent $\mathcal{D}_{01}\hat{\bm{u}}$ as the first order derivative in $C^{(1)}$ basis and $\mathcal{D}_{02}\hat{\bm{u}}$ as the second-order derivative in $C^{(2)}$ basis. In order to describe all the coefficients vector in the same basis, we need to multiply $\mathcal{S}_{12}$ to $\mathcal{D}_{01}\hat{\bm{u}}$, i.e. $\mathcal{S}_{12}\mathcal{D}_{01}\hat{\bm{u}}$ to convert the coefficients from $C^{(1)}$ basis to $C^{(2)}$ basis. Next, the coefficients of $a(x)u(x)$ in $C^{(2)}$ basis can be represented by $\mathcal{S}_{1,2}\mathcal{S}_{0,1}[\mathcal{M}_T[a]\hat{\bm{u}}]$ by (\ref{coe_mult}). Also, the coefficient vector of function $f(x)$ in $C^{(2)}$ basis can be written as $\mathcal{S}_{1,2}\mathcal{S}_{0,1}\hat{\bm{f}}$ which changes $\hat{\bm{f}}$ from $T$ basis to $C^{(2)}$ basis. Therefore, the equation in (\ref{eqn:bvp}) can be expressed by a linear matrix problem
\begin{align}\label{lin_mat}
    \mathcal{L}\hat{\bm{u}} = \mathcal{F},
\end{align}
where the matrix operator $\mathcal{L}$ and $\mathcal{F}$ are given by 
\begin{align*}
    \mathcal{L} = \mathcal{D}_{02} + \mathcal{S}_{12}\mathcal{D}_{01} + \mathcal{S}_{1,2}\mathcal{S}_{0,1}\mathcal{M}_T[a], \ \ \ \mathcal{F} = \mathcal{S}_{1,2}\mathcal{S}_{0,1}\hat{\bm{f}}.
\end{align*}
Furthermore, to impose the homogeneous Dirichlet and Neumann boundary conditions (\ref{eqn:bvp}), we remove the last two rows from $\mathcal{L}$ and $\mathcal{F}$, creating $\Bar{\mathcal{L}}$ and $\Bar{\mathcal{F}}$, respectively. For the homogeneous Dirichlet boundary condition, we observe that
\begin{align*}
 & u(1) = \sum_{k=0}^{N}\hat{u}_kT_k(1) = \sum_{k=0}^{N}\hat{u}_k = 0, \\
 & u(-1) = \sum_{k=0}^{N}\hat{u}_kT_k(-1) = \sum_{k=0}^{N}\hat{u}_k (-1)^k= 0.
\end{align*}
On the other hand, the homogeneous Neumann boundary condition yields
 \begin{align*}
 & u'(1) = \sum_{k=0}^{N}\hat{u}_k\left(\lim_{x\to1}T_k'(x)\right) = \sum_{k=0}^{N}\hat{u}_kk^2 = 0, \\
 & u'(-1) = \sum_{k=0}^{N}\hat{u}_k\left(\lim_{x\to-1}T_k'(x)\right) = \sum_{k=0}^{N}\hat{u}_k (-1)^{k+1}k^2= 0.
\end{align*}
This adjustment makes room for the boundaries as   
\begin{align}\label{eqn:mat_dir_neu}
    \begin{pmatrix}
         \Bar{\mathcal{L}}  \\
        \mathcal{B} 
    \end{pmatrix}\hat{\bm{u}} = 
    \begin{pmatrix}
        \Bar{\mathcal{F}} \\
\bm{0}
    \end{pmatrix},
\end{align}
where 
\begin{align}\label{eqn:B}
    \mathcal{B} = 
    \begin{bmatrix}  1 & 1 & 1 & \cdots & 1 \\
        1 & -1 & 1 & \cdots & (-1)^N  
    \end{bmatrix}, \text{ and }        
    \mathcal{B} = \begin{bmatrix}  0^2 & 1^2 & 2^2 & \cdots & N^2 \\
        -0^2 & 1^2 & -2^2 & \cdots & (-1)^{N+1}N^2   
    \end{bmatrix},
\end{align}     
denote the homogeneous Dirichlet and Neumann boundary conditions, respectively.

\subsection{Ultraspherical spectral method for Helmholtz equation in disk domain}\label{subsec:poisson}

Note that the key steps of solving pACOK (\ref{eqn:pACOK}) and pACNO (\ref{eqn:pACNO}) over disk domain is to find $(-\Delta)^{-1}u$ as defined in (\ref{eqn:LapInv}) and to solve for the equations as in (\ref{eqn:OK_fully_Scheme}, \ref{eqn:NO_fully_Scheme}). These involve solving the non-homogeneous Helmholtz equation over disk domain with the given boundary conditions in (\ref{eqn:LapInv}). We detail the numerical treatment as follows.

Consider the Helmholtz equation $-\Delta u + \alpha u  = f$ in polar coordinate, where $\alpha$ is a constant. We apply the disk analogue to the DFS method and extend the disk domain $\Omega = [0,1] \times [0,2\pi)$ to $\Tilde{\Omega} = [-1,1] \times [0,2\pi)$, with the extended Helmholtz equation  $-\Delta \tilde{u} + \alpha \tilde{u}  = \tilde{f}$. Here, $\tilde{f}$ is the BMC-II extension of $f$ from (\ref{def:BMC-II}), which satisfies the regularity over the origin of the unit disk and can be represented by     
\begin{align}\label{eqn:2Dpoisson}
   & -(r^2\tilde{u}_{rr}+r\tilde{u}_{r}+\tilde{u}_{\theta\theta}) + \alpha r^2\tilde{u}= r^2\tilde{f}, \ \ (r,\theta)\in \Tilde{\Omega} = [-1,1] \times [0,2\pi), \\
   & \tilde{u} \text{ is homogeneously Neumann in }r \text{ and periodic in }\theta,
\end{align}
where the multiplication by $r^2$ aims to deal with the singularity at $r=0$. The actual solution $u$ can be given by restricting $\tilde{u}$ to $\Omega = [0,1] \times [0,2\pi)$.

The solution $\tilde{u}$ in (\ref{eqn:2Dpoisson}) can be approximated by the Chebyshev-Fourier series as 
\begin{align}\label{eqn:CF_exp_BMCII}
    \tilde{u}(r,\theta) \approx \sum_{l=-\frac{N_{\theta}}{2}}^{\frac{N_{\theta}}{2}-1}\sum_{k=1}^{N_{r}}\hat{u}_{k,l}T_{k}(r) e^{\text{i}l\theta},
\end{align}
where $N_{\theta}$ is a positive even integer defined the number of uniform mesh points in $\theta$ and $N_{r}$ is a positive odd integer defined the number of Gaussian-Lobatto points in $r$. The Chebyshev-Fourier coefficients $\hat{u}_{k,l}$ can be solved by a system of equations from (\ref{eqn:2Dpoisson})
\begin{align}\label{eqn:coef_poisson}
   & - \sum_{k=0}^{N_{r}}\hat{u}_{k,l}\left[r^2T_{k}''(r) +rT_{k}'(r)-(l^2+\alpha)T_{k}(r)\right]e^{il\theta} = \sum_{k=0}^{N_{r}} \hat{f}_{k,l}r^2T_{k}(r)e^{il\theta}, \ \ \ l=-\frac{N_{\theta}}{2}:\frac{N_{\theta}}{2}-1.
\end{align}
Define a matrix operator $\mathcal{L}$ as 
\begin{align}
    \mathcal{L} = -\mathcal{M}_2[r^2]\mathcal{D}_{0,2} - \mathcal{S}_{1,2}\mathcal{M}_1[r]\mathcal{D}_{0,1} + (l^2+\alpha)\mathcal{S}_{0,2},  \ \ \ l=-\frac{N_{\theta}}{2}:\frac{N_{\theta}}{2}-1,
\end{align}
where $\mathcal{M}_2[r^2] = \mathcal{S}_{0,2}\mathcal{M}_T[r^2]\mathcal{S}_{2,0}$ is for the multiplication by $r^2$ in $C^{(2)}$ basis and $\mathcal{M}_1[r] = \mathcal{S}_{0,1}\mathcal{M}_T[r]\mathcal{S}_{1,0}$ is for the multiplication by $r$ in $C^{(1)}$ basis. Therefore, the linear system from (\ref{eqn:coef_poisson}) becomes 
\begin{align}\label{eqn:mat_poi_coef}
     \mathcal{L}
    \hat{\bm{u}}_{l} =  \mathcal{S}_{0,2}\mathcal{M}_T(r^2)\hat{\bm{f}}_{l}, \ \ l=-\frac{N_{\theta}}{2}:\frac{N_{\theta}}{2}-1.
\end{align}
where $\hat{\bm{u}}_l = [\hat{u}_{1,l},\hat{u}_{2,l},\cdots,\hat{u}_{N_r,l}]^{T}$ and $\hat{\bm{f}}_l = [\hat{f}_{1,l},\hat{f}_{2,l},\cdots,\hat{f}_{N_r,l}]^{T}$ as the vector of coefficients in Chebyshev-Fourier basis. Incorporating the homogeneous boundary condition $\mathcal{B}$ in (\ref{eqn:B})  gives us the system 
\begin{align}\label{eqn:lin_sym}
\begin{pmatrix}
          \mathcal{L}(1:\text{end}-2,:) \\
          \mathcal{B}
       \end{pmatrix} \hat{\bm{u}}_l = \begin{pmatrix}
           [\mathcal{S}_{0,2}\mathcal{M}_T(r^2)\hat{\bm{f}}_{l}](1:\text{end}-2,:)  \\
\bm{0}
\end{pmatrix}, \ \  l=-\frac{N_{\theta}}{2}:\frac{N_{\theta}}{2}-1.
\end{align}

There are two methods for solving the linear system (\ref{eqn:lin_sym}). The first one is the fast disk solver introduced by Townsend \cite{Wilber_Townsend_Wright2017}, which combines the alternating direction implicit (ADI) method with the ultraspherical spectral method, directly computing the low-rank approximations of the Chebyshev-Fourier coefficients $\hat{u}\in\mathbb{C}^{N_{r}\times N_{\theta}}$. The second method is to apply the Sherman–Morrison formula \cite{ShermanMorrison_AMS1950}, see details in \cite{Wilber_thesis2016}. This method solves the linear system (\ref{eqn:lin_sym}) separately for the even-indexed and odd-indexed terms of the Chebyshev-Fourier coefficients $\hat{u}_{2k,2l}$ and $\hat{u}_{2k+1,2l+1}$, reducing the computational complexity of the system. For the detailed discussion, we refer to \cite{Boyd2001,Fornberg1995,Trefethen2000,Shen2000}.

Wilber compared these two methods in \cite[Figure~7]{Wilber_Townsend_Wright2017}. For large mesh grid sizes, when the numerical rank of $\hat{u}$ is sufficiently low, the ADI method with low-rank approximation proves faster than the full-rank solver. However, for small to medium mesh sizes, when the rank is not sufficiently low, the second method is often preferred.

In our work, we often use medium-sized mesh of $512$ or $256$. Therefore, we apply the algorithm from \cite{Wilber_thesis2016} to the linear system (\ref{eqn:lin_sym}) with adjustments for the homogeneous Neumann boundary condition. Detailed application can be found in Section (\ref{subsec:FullyScheme}).

\section{Second Order Time-discrete Energy Stable Scheme}\label{sec3}

In this section, we aim to find the equilibrium state of the binary and ternary systems introduced by the OK and NO models in disk domain. By introducing the second order Backward Differentiation Formula (BDF) method \cite{Choi_Zhao2021} for temporal discretization and the ultraspherical spectral method for the spatial discretization for pACOK (\ref{eqn:pACOK}) and pACNO (\ref{eqn:pACNO}) equations on the unit disk, we study the energy stability for both semi- and fully-discrete schemes.

\subsection{Second order semi-discrete scheme for pACOK and pACON equation in disk domain}\label{Subsec:Semi}

For the OK model, we adopt the second-order BDF method from \cite{Choi_Zhao2021} for the pACOK equation (\ref{eqn:pACOK}) in the disk domain. For a given time interval $[0,T]$ and an integer $N>0$, we define a uniform time step size $\tau = \frac{T}{N}$ with $t_n = n\tau$ for $n = 0,1,\cdots,N$ and the approximate solution $u^n \approx u(r,\theta,t_n)$. Choosing the stabilizer 
\begin{align*}
    u^{n+1}-2u^{n}+u^{n-1},
\end{align*}
the second order BDF scheme for the pACOK equation follows: given initial data $u^{-1} = u^0 = u_0$, we can find $u^{n+1}$ for each $n = 1,2,\cdots,N$ such that
\begin{align}\label{eqn:pACOK_scheme}
    \frac{3u^{n+1}-4u^{n}+u^{n-1}}{2\tau}  = \ & \epsilon\Delta u^{n+1} - \frac{1}{\epsilon}\left[2W'(u^n)-W'(u^{n-1})\right] \nonumber\\
    & -\frac{\kappa}{\epsilon}\left(u^{n+1}-2u^{n}+u^{n-1}\right) -\gamma\beta(-\Delta)^{-1}\left(u^{n+1}-2u^{n}+u^{n-1}\right) \nonumber\\
    & -\gamma(-\Delta)^{-1}\left[2u^n-u^{n-1}-\omega\right] -M\left[\int_{\Omega}\left(2u^n-u^{n-1}\right)\ \text{d}x - \omega|\Omega|\right],
\end{align}
where $\frac{\kappa}{\epsilon}\left(u^{n+1}-2u^{n}+u^{n-1}\right)$ and $\gamma\beta(-\Delta)^{-1}\left(u^{n+1}-2u^{n}+u^{n-1}\right)$ control the energy stability of the system and $\kappa, \beta\geq0$ are stabilization constants.

For the NO model, the second order BDF scheme is used to explore the pACNO equation system in the disk domain over time interval $[0,T]$, and the uniform time step size $\tau = \frac{T}{N}$ with $t_n = n\tau$ for $n = 0,1,\cdots,N$ is defined by choosing an integer $N>0$. Let the approximate solutions $u^n_i \approx u(r,\theta,t_n)$ for $i=1,2$. Incorporating the stabilizer 
\begin{align*}
    u^{n+1}_i-2u^{n}_i+u^{n-1}_i,\ \ i = 1,2.
\end{align*} 
the second order BDF scheme for the pACON equation system becomes: given initial data $u_i^{-1} = u_i^0 = u_{i,0}$, the approximate solutions $u^{n+1}_i$, $i=1,2$ for $n = 1,2,\cdots,N$ can be found by 
\begin{align}\label{eqn:pACNO_scheme}
    &\frac{3u^{n+1}_i-4u^{n}_i+u^{n-1}_i}{2\tau} \nonumber\\ 
    =\ & \epsilon\Delta u^{n+1}_i + \frac{\epsilon}{2}\left(2\Delta u^{n+\text{mod}(i+1,2)}_{j}-\Delta u^{n-1+2\cdot\text{mod}(i+1,2)}_{j}\right) \nonumber \\
    & - \frac{1}{\epsilon}\left[2\frac{\partial W_2}{\partial u_{i}}(u^{n+\text{mod}(i+1,2)}_1,u^{n}_{2})-\frac{\partial W_2}{\partial u_{i}}(u^{n-1+2\cdot\text{mod}(i+1,2)}_1,u^{n-1}_{2})\right] \nonumber \\
    & -\frac{\kappa_i}{\epsilon}\left(u^{n+1}_i-2u^{n}_i+u^{n-1}_i\right) - \gamma_{ii}\beta_{i}(-\Delta)^{-1}\left(u^{n+1}_i-2u^{n}_i+u^{n-1}_i\right)  \nonumber\\
    & -\gamma_{ii}(-\Delta)^{-1}\left[2u^n_i-u^{n-1}_i-\omega_i\right]  -\gamma_{ij}(-\Delta)^{-1}\left[2u^{n+\text{mod}(i+1,2)}_j-u^{n-1+2\cdot\text{mod}(i+1,2)}_j-\omega_j\right]\nonumber\\
    & -M_i\left[\int_{\Omega}\left(2u^n_i-u^{n-1}_i\right)\ \text{d}x - \omega_i|\Omega|\right],
\end{align}
for $i = 1,2$ and $j\neq i$, and $\kappa_1,\kappa_2,\beta_1,\beta_1\geq0$ are stabilization constants.

Next, we can ensure that the second-order BDF schemes, (\ref{eqn:pACOK_scheme}) for pACOK and (\ref{eqn:pACNO_scheme}) for pACNO, always have a unique solution, as shown by the following lemma.

\begin{lemma}\label{thm:Solvable_time_dicrete} 
The second-order BDF schemes (\ref{eqn:pACOK_scheme}) for the pACOK equation and (\ref{eqn:pACNO_scheme}) for the pACNO equations at time-discrete level are uniquely solvable.
\end{lemma}
\begin{proof}
Since two schemes are similar, we only need to prove for the scheme (\ref{eqn:pACOK_scheme}) of pACOK. This scheme can be rewritten as 
\begin{align*}
\left(\dfrac{3}{2\tau}+\dfrac{\kappa}{\epsilon}+\epsilon(-\Delta)+\gamma\beta(-\Delta)^{-1}\right)u^{n+1} = \text{RHS},
\end{align*}
where $\text{RHS}$ is the nonlinear terms and only depends on $u^{n}$ and $u^{n-1}$. Then in the disk domain under the homogeneous Neumann boundary condition, the operator $(-\Delta)$ is positive-definite by the Green's identity. Therefore, the operator applying on $u^{n+1}$ is invertible, which shows the unique solvability of the schemes (\ref{eqn:pACOK_scheme}) and (\ref{eqn:pACNO_scheme}).
\end{proof}

Then, for the energy stability of the pACOK and pACNO equations (\ref{eqn:pACOK_scheme}, \ref{eqn:pACNO_scheme}), we have the following theorem.

\begin{theorem}\label{thm:OK/NO_energy} 
For pACOK equation, we define a modifying energy functional $\tilde{E}^{\text{pOK}}[u^{n},u^{n-1}]$ for the OK model
 \begin{align}\label{eqn:OK_modifiedenergy}
\tilde{E}^{\text{pOK}}(u^{n},u^{n-1} ) = E^{\text{pOK}}(u^n)+\left(\frac{\kappa}{2\epsilon}+\frac{1}{4\tau}+C\right)\|u^{n}-u^{n-1}\|_{L^2}^2+\frac{\gamma \beta}{2}\|(-\Delta)^{-\frac{1}{2}}\left(u^{n}-u^{n-1}\right)\|_{L^2}^2,
\end{align}
where $\kappa,\beta\geq0$, and the constant $C$ is given by 
\begin{align}\label{cst:pACOK_cont}
    C=\frac{L_{W''}}{2\epsilon}+\frac{\gamma}{2}\|(-\Delta)^{-1}\|_{L^2}+\frac{M}{2}|\Omega|.
\end{align}
Let $\{u^{n}\}_{n=1}^N$ be generated by scheme (\ref{eqn:pACOK_scheme}), and the time step size $\tau\leq\frac{1}{3C}$, then
\begin{align*}
    \tilde{E}^{\text{pOK}}[u^{n+1},u^{n}]\leq  \tilde{E}^{\text{pOK}}[u^{n},u^{n-1}].
\end{align*}
For pACNO equations,we define the modifying energy functional $ \tilde{E}^{\text{pNO}}[u_1^{n},u_1^{n-1},u_2^{n},u_2^{n-1}]$ as 
\begin{align}\label{eqn:mod_energy_pNO}
     \tilde{E}^{\text{pNO}}[u_1^{n},u_1^{n-1},u_2^{n},u_2^{n-1}]=& E^{\text{pNO}}[u_1^{n},u_2^{n}] \nonumber\\
    & +\sum_{i=1}^{2}\left[\left(\frac{\kappa_i}{2}+\frac{1}{4\tau}+C_i\right)\|u_i^{n}-u_i^{n-1}\|_{L^2}^2+\frac{\gamma_{ii} \beta_{i}}{2}\|(-\Delta)^{-\frac{1}{2}}(u_i^{n}-u_i^{n-1})\|^2_{L^2}\right],
\end{align}
where $\kappa_i,\beta_i\geq0$, $i=1,2$, and the constants $C_i$ are given by 
\begin{align}\label{cst:pNO}
C_i = \frac{L_{W''}}{2\epsilon}+\frac{\gamma_{i1}+\gamma_{i2}}{2} \|(-\Delta)^{-1}\|_{L^2}+\frac{M_i}{2}|\Omega|.
\end{align}
Let $\{u_1^{n},u_2^{n}\}_{n=1}^N$ be generated by scheme (\ref{eqn:pACNO_scheme}), and the time step size $\tau\leq \min\{\frac{1}{3C_1},\frac{1}{3C_2}\}$, then
\begin{align*}
     \tilde{E}^{\text{pNO}}[u_1^{n+1},u_1^{n},u_2^{n+1},u_2^{n}]\leq  \tilde{E}^{\text{pNO}}[u_1^{n},u_1^{n-1},u_2^{n},u_2^{n-1}],
\end{align*}
\end{theorem}

\begin{proof}
The proofs of the energy stability for the pACOK and pACNO equations are similar. The main difference is that for the pACNO schemes, we deal with $i=1,2$ separately and then add them up. Therefore, we omit the proof for the pACOK equation and provide a detailed proof for the pACNO equations. Note that this proof is similar to the one for the square case by Zhao and Choi \cite{Choi_Zhao2021}.

By taking $L^2$ inner product of the two equations in (\ref{eqn:pACNO_scheme}) with $u_{i}^{n+1}-u_{i}^n$ for $i=1,2$, and using the identities $a\cdot(a-b)=\frac{1}{2}a^2-\frac{1}{2}b^2+\frac{1}{2}(a-b)^2$ and $b\cdot(a-b)=\frac{1}{2}a^2-\frac{1}{2}b^2-\frac{1}{2}(a-b)^2$, with $a=u_i^{n+1}-u_i^{n}$, $b=u_i^{n}-u_i^{n-1}$, we obtain two equations:
\begin{align}
   & \frac{1}{\tau}\|u_{1}^{n+1}-u_{1}^{n}\|_{L^2}^2+\frac{1}{4\tau}\left(\|u_{1}^{n+1}-u_{1}^{n}\|_{L^2}^2-\|u_{1}^{n}-u_{1}^{n-1}\|_{L^2}^2+\|u_{1}^{n+1}-2u_{1}^{n}+u_{1}^{n-1}\|_{L^2}^2\right) \nonumber \\
 = & -\frac{\epsilon}{2}\left(\left\langle-\Delta u_{1}^{n+1},u_{1}^{n+1}\right\rangle-\left\langle-\Delta u_{1}^{n},u_{1}^{n}\right\rangle +\left\langle-\Delta(u_{1}^{n+1}-u_{1}^{n}),u_{1}^{n+1}-u_{1}^{n}\right\rangle + \left\langle -\Delta \left(2u_{2}^{n} -  u_{2}^{n-1}\right),u_{1}^{n+1}-u_{1}^{n}\right\rangle \right) \nonumber\\
    & \underbrace{- \frac{1}{\epsilon}\left\langle 2\frac{\partial W_2}{\partial u_1}(u_1^n, u_2^n) - \frac{\partial W_2}{\partial u_1}(u_1^{n-1}, u_2^{n-1}), u_1^{n+1} - u_{1}^{n}\right\rangle}_{\text{I(a)}} \nonumber\\
    & -\frac{\kappa_{1}}{2\epsilon}\left(\|u_{1}^{n+1}-u_{1}^n\|_{L^2}^2+\|u_{1}^{n}-u_{1}^{n-1}\|_{L^2}^2-\|u_{1}^{n+1}-2u_{1}^n+u_{1}^{n-1}\|_{L^2}^2 \right) \nonumber\\
    & -\frac{\gamma_{11} \beta_{1}}{2}\left(\|(-\Delta)^{-\frac{1}{2}}(u_1^{n+1}-u_1^{n})\|_{L^2}^2-\|(-\Delta)^{-\frac{1}{2}}(u_1^{n}-u_1^{n-1})\|_{L^2}^2 + \|(-\Delta)^{-\frac{1}{2}}(u_1^{n+1}-2u_1^{n}+u_1^{n-1})\|_{L^2}^2\right) \nonumber\\
    & \underbrace{-\gamma_{11}\left\langle(-\Delta)^{-1}\left(2u_1^{n}-u_1^{n-1}-\omega_1\right),u_1^{n+1}-u_1^{n}\right\rangle}_{\text{II(a)}}\underbrace{ -\gamma_{12}\left\langle(-\Delta)^{-1}\left(2u_2^{n}-u_2^{n-1}-\omega_2\right),u_1^{n+1}-u_1^{n}\right\rangle }_{\text{III(a)}}\nonumber\\
    & \underbrace{-M_1\left(\int_{\Omega} (2u_1^n-u_1^{n-1})\ \text{d}A - \omega_1|\Omega|\right)\left\langle 1,u_1^{n+1}-u_1^{n}\right\rangle}_{\text{IV(a)}}; \label{eqn:pACNO_i=1} 
    \end{align}
    \begin{align}
    & \frac{1}{\tau}\|u_{2}^{n+1}-u_{2}^{n}\|_{L^2}^2+\frac{1}{4\tau}\left(\|u_{2}^{n+1}-u_{2}^{n}\|_{L^2}^2-\|u_{2}^{n}-u_{2}^{n-1}\|_{L^2}^2+\|u_{2}^{n+1}-2u_{2}^{n}+u_{2}^{n-1}\|_{L^2}^2\right) \nonumber\\
 = & -\frac{\epsilon}{2}\left(\left\langle-\Delta u_{2}^{n+1},u_{2}^{n+1}\right\rangle-\left\langle-\Delta u_{2}^{n},u_{2}^{n}\right\rangle+\left\langle-\Delta(u_{2}^{n+1}-u_{2}^{n}),u_{2}^{n+1}-u_{2}^{n}\right\rangle+\left\langle -\Delta u_{1}^{n+1} ,u_{2}^{n+1}-u_{2}^{n}\right\rangle \right) \nonumber\\
    &  \underbrace{- \frac{1}{\epsilon}\left\langle 2\frac{\partial W_2}{\partial u_2}(u_1^{n+1}, u_2^n) - \frac{\partial W_2}{\partial u_2}(u_1^{n+1}, u_2^{n-1}), u_2^{n+1} - u_{2}^{n}\right\rangle}_{I(b)} \nonumber\\
    & -\frac{\kappa_{2}}{2\epsilon}\left(\|u_{2}^{n+1}-u_{2}^n\|_{L^2}^2+\|u_{2}^{n}-u_{2}^{n-1}\|_{L^2}^2-\|u_{2}^{n+1}-2u_{2}^n+u_{2}^{n-1}\|_{L^2}^2 \right) \nonumber\\
    & -\frac{\gamma_{22} \beta_{2}}{2}\left(\|(-\Delta)^{-\frac{1}{2}}(u_2^{n+1}-u_2^{n})\|_{L^2}^2-\|(-\Delta)^{-\frac{1}{2}}(u_2^{n}-u_2^{n-1})\|_{L^2}^2 + \|(-\Delta)^{-\frac{1}{2}}(u_2^{n+1}-2u_2^{n}+u_2^{n-1})\|_{L^2}^2\right) \nonumber\\
    & \underbrace{-\gamma_{22}\left\langle(-\Delta)^{-1}\left(2u_2^{n}-u_2^{n-1}-\omega_2\right),u_2^{n+1}-u_2^{n}\right\rangle}_{\text{II(b)}} \underbrace{-\gamma_{21}\left\langle(-\Delta)^{-1}\left(u_1^{n+1}-\omega_1\right),u_2^{n+1}-u_2^{n}\right\rangle}_{\text{III(b)}} \nonumber\\
    & \underbrace{-M_2\left(\int_{\Omega} (2u_2^n-u_2^{n-1})\ \text{d}A - \omega_2|\Omega|\right)\left\langle 1,u_2^{n+1}-u_2^{n}\right\rangle}_{\text{IV(b)}}. \label{eqn:pACNO_i=2}
\end{align}
Next, we add the terms of equations (\ref{eqn:pACNO_i=1}) and (\ref{eqn:pACNO_i=2}) one by one. The combined terms II(a)+II(b) and IV(a)+IV(b) can be written as
\begin{align*}
   \text{II(a)} + \text{II(b)} 
     = & -\sum_{i=1}^{2}\frac{\gamma_{ii}}{2}\left(\left\langle(-\Delta)^{-1}[u^{n+1}-\omega],u^{n+1}-\omega\right\rangle-\left\langle(-\Delta)^{-1}[u^{n}-\omega],u^{n}-\omega\right\rangle\right) \\
     & +\sum_{i=1}^{2}\frac{\gamma_{ii}}{2}\left(\left\langle(-\Delta)^{-1}[u^{n+1}-u^n],u^{n+1}-u^n\right\rangle - 2\left\langle(-\Delta)^{-1}[u^n-u^{n-1}],u^{n+1}-u^n\right\rangle\right)\\
     \leq & -\sum\limits_{i=1}^{2}\frac{\gamma_{ii}}{2}\left(\left\langle(-\Delta)^{-1}\left(u_i^{n+1}-\omega_i\right),u_i^{n+1}-\omega_i\right\rangle -\left\langle(-\Delta)^{-1}\left(u_i^{n}-\omega_i\right),u_i^{n}-\omega_i\right\rangle \right)\\
 & + \sum\limits_{i=1}^{2}\frac{\gamma_{ii}}{2} \|(-\Delta)^{-1}\|_{L^2} \left(\|u^{n}-u^{n-1}\|_{L^2}^2+2\|u^{n+1}-u^n\|_{L^2}^2\right), \\
      \text{IV(a)} + \text{IV(b)} 
     = & -\sum_{i=1}^{2}\frac{M_i}{2} \left(\left\langle u^{n+1} - \omega,1\right\rangle^2 - \left\langle u^n -\omega ,1\right\rangle^2\right)\\
        & +\sum_{i=1}^{2}\frac{M_i}{2} \left(\left\langle 1,u^{n+1}-u^{n}\right\rangle^2-2\left\langle u^n-u^{n-1},1\right\rangle \left\langle 1,u^{n+1}-u^{n}\right\rangle\right) \\
        \leq & -\sum\limits_{i=1}^{2} \frac{M_{i}}{2} \left(\left\langle u_{i}^{n+1} - \omega_i,1\right\rangle^2 - \left\langle u_i^n -\omega_i ,1\right\rangle^2\right)  \\
        &+ \sum\limits_{i=1}^{2} \frac{M_i}{2}|\Omega|\left(\|u_i^{n}-u_i^{n-1}\|_{L^2}^2+2\|u_i^{n+1}-u_i^n\|_{L^2}^2\right).
\end{align*}
The other terms, I(a)+I(b) and III(a)+III(b), which include the coupling terms between $u_1$ and $u_2$, can be estimated by Taylor expansion and Young's inequality. The estimations state as following:
\begin{align*}
   & \text{I(a)} + \text{I(b)} \\
\leq & -\frac{1}{\epsilon}\left\langle W_2(u_1^{n+1},u_2^{n+1}),1\right\rangle + \frac{1}{\epsilon}\left\langle W_2(u_1^{n},u_2^{n}),1\right\rangle + \frac{L_{W''}}{2\epsilon}\sum\limits_{i=1}^{2}\left(\|u^{n}-u^{n-1}\|_{L^2}^2+2\|u^{n+1}-u^n\|_{L^2}^2\right);\\
& \text{III(a)} + \text{III(b)} \\
\leq & -\gamma_{21}\left\langle (-\Delta)^{-\frac{1}{2}}\left(u_2^{n+1}-\omega_2\right),(-\Delta)^{-\frac{1}{2}}\left(u_1^{n+1}-\omega_1\right)\right\rangle + \gamma_{12}\left\langle (-\Delta)^{-\frac{1}{2}}\left(u_1^{n}-\omega_1\right),(-\Delta)^{-\frac{1}{2}}\left(u_2^{n}-\omega_2\right)\right\rangle \\
& + \frac{1}{2} \|(-\Delta)^{-1}\|_{L^2} \left(\gamma_{21}\|u_2^{n}-u_{2}^{n-1}\|_{L^2}^2 + \gamma_{12}\|u_1^{n+1}-u_{1}^{n}\|_{L^2}^2\right).
\end{align*} 
Note that the energy functional (\ref{eqn:pNO}) of the NO model can be written as
\begin{align*}
E^{\text{pNO}}(u_1^{n},u_2^{n}) = & \frac{\epsilon}{2}\left(\left\langle-\Delta u_1^{n},u_1^{n}\right\rangle+\left\langle-\Delta u_2^{n},u_2^{n}\right\rangle+\left\langle \nabla u_1^{n},\nabla u_2^{n}\right\rangle\right)+\frac{1}{\epsilon}\left\langle W_2(u_1^{n},u_2^{n}),1 \right\rangle \\
& + \sum\limits_{i,j=1}^{2}\frac{\gamma_{ij}}{2}\left\langle(-\Delta)^{-\frac{1}{2}}\left(u_i^{n}-\omega_i\right),(-\Delta)^{-\frac{1}{2}}\left(u_j^{n}-\omega_j\right)\right\rangle + \sum\limits_{i=1}^{2}\frac{M_i}{2} \left\langle u_i^{n}-\omega_i, 1 \right\rangle^2,
\end{align*}
Now, rewriting the summation and drop the unnecessary terms, we can get the following inequality
\begin{align*}
&\left[E^{\text{pNO}}(u_1^{n+1},u_2^{n+1})+\sum\limits_{i=1}^{2}\left(\frac{\kappa_i}{2\epsilon}+\frac{1}{4\tau}\right)\|u_i^{n+1}-u_i^n\|_{L^2}^2+\sum\limits_{i=1}^{2}\frac{\gamma_{ii} \beta_i}{2}\left\langle(-\Delta)^{-1}(u_i^{n+1}-u_i^{n}),u_i^{n+1}-u_i^n\right\rangle\right] \\
& -\left[E^{\text{pNO}}(u_1^{n},u_2^{n})+\sum\limits_{i=1}^{2}\left(\frac{\kappa_i}{2\epsilon}+\frac{1}{4\tau}\right)\|u_i^{n}-u_i^{n-1}\|_{L^2}^2+\sum\limits_{i=1}^{2}\frac{\gamma_{ii} \beta_i}{2}\left\langle(-\Delta)^{-1}(u_i^{n}-u_i^{n-1}),u_i^{n}-u_i^{n-1}\right\rangle\right] \\
 \leq &\sum\limits_{i=1}^{2}\left(\frac{L_{W''}}{2\epsilon}+\frac{\gamma_{ii}}{2} \|(-\Delta)^{-1}\|_{L^2}+\frac{M_i}{2}|\Omega|\right)\left(\|u_i^{n}-u_i^{n-1}\|_{L^2}^2+2\|u_i^{n+1}-u_i^n\|_{L^2}^2\right) \\
 & + \frac{1}{2} \|(-\Delta)^{-1}\|_{L^2} \left(\gamma_{21}\|u_2^{n}-u_{2}^{n-1}\|_{L^2}^2 + \gamma_{12}\|u_1^{n+1}-u_{1}^{n}\|_{L^2}^2\right) - \sum\limits_{i=1}^{2}\frac{1}{\tau}\|u_i^{n+1}-u_{i}^{n}\|_{L^2}^2 \\
 \leq & \sum\limits_{i=1}^{2} \left(C_i\left(\|u_i^{n}-u_i^{n-1}\|_{L^2}^2+2\|u_i^{n+1}-u_i^n\|_{L^2}^2\right) - \frac{1}{\tau}\|u_i^{n+1}-u_{i}^{n}\|_{L^2}^2\right)
\end{align*}
where $C_i$ are given by
\begin{align}\label{cst:pNO}
C_i = \frac{L_{W''}}{2\epsilon}+\frac{\gamma_{i1}+\gamma_{i2}}{2} \|(-\Delta)^{-1}\|_{L^2}+\frac{M_i}{2}|\Omega|,
\end{align}
where $L_{W''}$ and $\|(-\Delta)^{-1}\|_{L^2}$ are constants which can be found in Section (\ref{subsec:notation}). 

Adding $ \sum\limits_{i=1}^{2}C_i\left(\|u_i^{n+1}-u_i^n\|_{L^2}^2-\|u_i^{n}-u_i^{n-1}\|_{L^2}^2\right)$ to both side of the above inequality, we can find that 
\begin{align*}
 \tilde{E}^{\text{pNO}}[u_1^{n+1},u_1^{n},u_2^{n+1},u_2^{n}] - \tilde{E}^{\text{pNO}}[u_1^{n},u_1^{n-1},u_2^{n},u_2^{n-1}] \leq  \sum\limits_{i=1}^{2} \left(3C_i-\frac{1}{\tau}\right)\|u_i^{n+1}-u_i^n\|_{L^2}^2,
\end{align*}
where the modified energy functional for the pACNO defined in (\ref{eqn:mod_energy_pNO}). Therefore, if $3C_i-\frac{1}{\tau} \leq 0$, which mean $\tau\leq\min\{\frac{1}{3C_1},\frac{1}{3C_2}\}$, then 
\begin{align*}
     \tilde{E}^{\text{pNO}}[u_1^{n+1},u_1^{n},u_2^{n+1},u_2^{n}]\leq  \tilde{E}^{\text{pNO}}[u_1^{n},u_1^{n-1},u_2^{n},u_2^{n-1}].
\end{align*}
and the energy stability for the pACNO system is proved.
\end{proof}

\subsection{Fully-discrete scheme for pACOK and pACNO equations in the disk domain}
In this section, we apply the ultraspherical spectral method in the disk domain to construct the fully-discrete schemes for both pACOK and pACNO equations.

\subsubsection{Ultraspherical spectral method for spatial discretization}\label{subsection:spatial discretization}
For the pACOK/pACNO equations in the disk domain $\Omega = [0,1]\times [0,2\pi)$, we use the Chebyshev-Fourier expansion in the spatial domain to express the function $f$ as:
\begin{align}\label{eqn:Cheb_Four_Exp}
    f(r,\theta) = \sum_{j,k}\hat{f}_{jk}T_k(r)e^{\text{i}j\theta}, 
\end{align}
where $T_k(r)$ are the Chebyshev polynomials of the first kind. Next, instead of $\Omega = [0,1]\times [0,2\pi)$, we focus on an extended domain $\Tilde{\Omega}= [-1,1]\times [0,2\pi)$ which alleviates oversampling near the origin and reduces the computational complexity by even-odd parity.

We discretize the spatial operators using the spectral collocation approximation. Let $N_r$ be a positive odd integer and $N_{\theta}$ be a positive even integer, the collocation points are Chebyshev-Gaussian-Lobatto points $r_i = \cos{\frac{i\pi}{N_r}}$ with $i = 0,1,\cdots,N_r$ in $r$-direction and the uniform mesh in $\theta$-direction with mesh size $h = \frac{2\pi}{N_{\theta}}$. The set of collocation points is formulated as $\Tilde{\Omega}_{h} = \{(r_i,\theta_j); \ i = 0:N_{r},j = 1:N_{\theta}\}$. Then we define the index set:
\begin{align*}
    & S_{h} = \{(i,j)\in\mathbb{Z}^2;\ i = 0:N_{r},j = 1:N_{\theta}\}, \\
    & \hat{S}_h = \{(k,l)\in\mathbb{Z}^2;\ k = 0:N_{r},l = -\frac{N_{\theta}}{2}+1:\frac{N_{\theta}}{2}\}.
\end{align*}
Denote $\mathcal{M}_{h}$ as the collection of grid functions with zero mean defined on $\Tilde{\Omega}_{h}$:
\begin{align*}
    \mathcal{M}_{h} = \{f:\Tilde{\Omega}_{h}\to \mathbb{R} | \ f_{i,j+nN_{\theta}} = f_{i,j}, \forall(i,j)\in S_{h}, \forall n\in\mathbb{Z}, \ \hat{f}_{00} = 0\}.
\end{align*}
For any $f,g\in \mathcal{M}_{h}$ and $\mathbf{f} = (f^1,f^2)^T, \mathbf{g} = (g^1,g^2)^T\in \mathcal{M}_{h}\times\mathcal{M}_{h}$, we define the discrete $L^2$ inner product $\langle\cdot,\cdot\rangle_{h}$, discrete $L^2$-norm $||\cdot||_{L^2,h}$, and discrete $L^{\infty}$-norm $||\cdot||_{L^{\infty},h}$ as follows:
\begin{align*}
   & \langle f,g \rangle_{h} = h_{\theta}\sum_{(i,j)\in S_h}f_{ij}g_{ij}\omega_{i}, \ \ ||f||_{L^2,h} = \sqrt{\langle f,f \rangle_{h}}, \ \ ||f||_{L^{\infty},h} = \max_{(i,j)\in S_h}|f_{ij}|, \\
   &  \langle \mathbf{f},\mathbf{g} \rangle_{h} = h_{\theta}\sum_{(i,j)\in S_h}(f^1_{ij}g^1_{ij}+f^2_{ij}g^2_{ij})\omega_{i}, \ \ ||\mathbf{f}||_{L^2,h} = \sqrt{\langle \mathbf{f},\mathbf{f} \rangle_{h}}.
\end{align*}
where $\omega_{i}$ is the Chebyshev weights such that $\omega_{i} = \frac{\pi}{\Tilde{c}_{i}N_r}, \ i = 0:N_r$ and $\Tilde{c}_{0}=\Tilde{c}_{N_r}=2, \ \Tilde{c}_{i}=1, i=1:N_r-1$. Note that $\omega_i$ in the model formulation (\ref{eqn:pACNO}), with a slight abuse of notation, represents the volume fraction of $u_i$. The specific meaning of this notation can be determined by the context.

The 2D discrete Chebyshev-Fourier transform of a function $f\in \mathcal{M}_h$ is defined as
\begin{align*}
    \hat{f}_{kj} = \frac{2}{\Tilde{c}_kN_rN_{\theta}} \sum_{(i,j)\in S_h}\frac{f_{ij}}{\Tilde{c}_j}T_k(r_i)e^{-\text{i}l\theta_j}, \ \ (k,l)\in\hat{S}_h,
\end{align*}
and its inverse transform is given by
\begin{align*}
    f_{ij} = \sum_{(k,l)\in \hat{S}_h} \hat{f}_{kl}T_k(r_i)e^{\text{i}l\theta_j}, \ \ (i,j)\in S_h.
\end{align*}
Moreover, the operators in $\theta$-direction on Fourier basis can be found in \cite{Choi_Zhao2021}, in $r$-direction on Chebyshev basis are introduced in section (\ref{sec2}), and the spectral approximations to the differential operators $\partial_{\theta\theta}$, $\partial_{r}$, $\partial_{rr}$ follows:
\begin{align*}
   & \{f_{\theta\theta}\}_{ij} = \{\partial_{\theta\theta} f\}_{ij} = \sum_{(k,l)\in \hat{S}_h} -l^2\hat{f}_{kl}T_k(r_i)e^{\text{i}l\theta_j}, \ \ (i,j)\in S_h, \\
   & \{f_{r}\}_{ij} = \{\partial_{r} f\}_{ij} = \sum_{(k,l)\in \hat{S}_h} \hat{f}_{kl}T_k'(r_i)e^{\text{i}l\theta_j}, \ \ (i,j)\in S_h, \\
   & \{f_{rr}\}_{ij} = \{\partial_{rr} f\}_{ij} = \sum_{(k,l)\in \hat{S}_h} \hat{f}_{kl}T_k''(r_i)e^{\text{i}l\theta_j}, \ \ (i,j)\in S_h, 
\end{align*}
In matrix form, let $\hat{\mathbf{F}}_l = \{\hat{f}_{kl}\}_{k=0:N_{r}} $ be the vector of Chebyshev-Fourier coefficients with $l = -\frac{N_{\theta}}{2}+1:\frac{N_{\theta}}{2}$. By converting the Chebyshev basis to the ultraspherical basis, we have 
\begin{align*}
& \partial_{\theta\theta}: \hat{\mathbf{F}} \longrightarrow l^2\mathcal{S}_{0,2}\hat{\mathbf{F}}, \\
& \partial_{r}: \hat{\mathbf{F}} \longrightarrow \mathcal{S}_{1,2}\mathcal{D}_{0,1}\hat{\mathbf{F}}, \\
& \partial_{rr}: \hat{\mathbf{F}} \longrightarrow \mathcal{D}_{0,2}\hat{\mathbf{F}}.
\end{align*}
where $\mathcal{S}_{0,2}$, $\mathcal{S}_{1,2}$, $\mathcal{D}_{0,1}$, and $\mathcal{D}_{0,2}$ can be found in Section (\ref{sec2}).

\subsubsection{Estimation of $\|(-\Delta_h)^{-1}\|_{L^2,h}$}

In this subsection, we will provide an upper limit for the $L^2$ norm of the discrete inverse Laplacian operator $(-\Delta_h)^{-1}$, denoted by $\|(-\Delta_h)^{-1}\|_{L^2,h}$, which satisfies $\|(-\Delta_h)^{-1}f\|_{L^2,h}\leq C\|f\|_{L^2,h}$, where $f\in\mathcal{M}_h$ and the constant $C$ is a generic constant independent of the mesh size. Similar to the semi-discrete case (\ref{Subsec:Semi}), we require the the optimal constant $\|(-\Delta_h)^{-1}\|_{L^2,h}$ to guarantee the energy stability for the fully-discrete schemes (\ref{eqn:BDF_OK_fully}, \ref{eqn:BDF_NO_fully}).The following Lemma provides the stability of the discrete inverse Laplacian operator $(-\Delta_h)^{-1}$.

\begin{lemma}
For the function $f(r,\theta)\in\mathcal{M}_h$, the $L^2$-bound of the discrete inverse Laplacian operator is given by
\begin{align*}
\|(-\Delta_h)^{-1}\|_{L^2,h}  \leq C,
\end{align*}
with a generic constant $C$ independent of mesh size in both $r$- and $\theta$-directions.
\end{lemma}

\begin{proof}
The proof of this lemma is based on \cite[Chapter~6,7]{Canuto_Springer2006} and the parity property of the Cartesian and polar coordinates in the disk domain\cite{Boyd2001, Shen2000, EISEN_HEINRICHS_WITSCH1991, Boyd_Yu2011}. We present a brief outline of the proof as below. For the functions $f, u \in \mathcal{M}_h$, the discrete Laplacian operator $-\Delta_{h}$ satisfies:
\begin{align*} 
(-\Delta_h)u = f \ \ \ \Longrightarrow \ \ -(ru_{r})_r-\frac{1}{r}u_{\theta\theta} = rf.
\end{align*}
Here the derivatives for $u\in\mathcal{M}_h$ are defined in section \ref{subsection:spatial discretization}.
Taking the discrete inner product with $u$, we have
\begin{align*}
\left\langle rf,u\right\rangle_{h} = a_{h}(u,u) & = \langle-(ru_r)_r-\frac{1}{r}  u_{\theta\theta},u\rangle_{h}  \\
& = h_{\theta}\sum_{(i,j)\in S_{h}} \left[\left(-(ru_r)_r\right)_{ij} + \left(-\frac{1}{r}u_{\theta\theta}\right)_{ij}\right]u_{ij}\omega_{i}. 
\end{align*}
Next, we define 
\begin{align*}
b(u,u) & =  \int_{-1}^{1} -(ru_r)_ru\omega \ \text{d}r + \int_{-1}^{1} -\frac{1}{r}u_{\theta\theta}u\omega \ \text{d}r,
\end{align*}
with $\omega = (1-r^2)^{-\frac{1}{2}}$, which leads to $\omega_r = r\omega^3 = r(1-r^2)\omega^5$, $\omega_{rr} = (1+2r^2)\omega^5 $, and $\frac{\omega_{rr}}{2} - \omega_{r}^2\omega^{-1} = \frac{1}{2}\omega^5$. Note that, with a slight abuse of notation, $\omega$ is exclusively used to represent the weight function in this proof. In other sections of this paper, $\omega$ always denotes the volume fraction of $u$ for the OK model.

Using the techniques in \cite[Chapter~7]{Canuto_Springer2006}, we have the following two inequalities:
\begin{align}
& b(u,u) \geq  \int_{-1}^{1}\frac{1}{r}u_{\theta}^2\omega\ \text{d}r - \frac{3}{2}  \int_{-1}^{1} ru^2\omega^5 \ \text{d}r, \label{eqn:ineq_I} \\
& b(u,u) \geq \frac{1}{2}\int_{-1}^{1}ru^2\omega^5\ \text{d}r. \label{eqn:ineq_II}
\end{align}
Thus, the inequality (\ref{eqn:ineq_I}) and (\ref{eqn:ineq_II}) imply
\begin{align}
b(u,u) \geq  \int_{-1}^{1}\frac{1}{r}u_{\theta}^2\omega\ \text{d}r - 3 b(u,u) \ \ \ \Longrightarrow\ \ \int_{-1}^{1}u_{\theta}^2\omega\ \text{d}r \leq 4b(u,u). \nonumber
\end{align}
Thanks to the closed relation in \cite[Chapter~2, (2.2.17)]{Canuto_Springer2006}
\begin{align*}
\sum_{j=0}^Np(x_j)\omega_{j} = \int_{-1}^{1}p(x)\omega(x)\ \text{d}x, \ \ \forall p\in\mathbb{P}_{2N-1},
\end{align*}
where $\mathbb{P}_{2N-1}$ is the set of polynomials of degree $2N-1$, we obtain
\begin{align*}
\left\langle u_{\theta}, u_{\theta}\right\rangle_h = h_{\theta}\sum_{(i,j)\in S_{h}} \left(u_{\theta}\right)^2_{ij}\omega_{i} \leq 4 a_{h}(u,u).
\end{align*}
For $f,u \in \mathcal{M}_{h}$, using the discrete Poincare inequality for the Fourier series, we have 
\begin{align*}
\|u\|_{L^2,h} \leq C'\|u_{\theta}\|_{L^2,h},
\end{align*}
where $C'$ is a generic constant independent of the mesh size $h$. Therefore, we have the following inequalities
\begin{align*}
\|u\|_{L^2,h}^2 \leq C'\|u_{\theta}\|_{L^2,h}^2 \leq 4C'a_{h}(u,u)\leq 4C'\|f\|_{L^2,h}\|u\|_{L^2,h},
\end{align*}
which lead to the result
\begin{align*}
\|u\|_{L^2,h} = \|(-\Delta_h)^{-1}f\|_{L^2,h} \leq C\|f\|_{L^2,h},
\end{align*}
where $C$ is a generic constant independent of the mesh size in both $r$- and $\theta$-directions.
\end{proof}

\subsubsection{Fully-discrete schemes and their energy stability}\label{subsec:FullyScheme}
In this subsection, we present the fully-discrete schemes for the pACOK and pACNO equations in the disk domain. The numerical solutions are given by:
\begin{align*}
    U^n \approx u(r,\theta;t_n)|_{\Tilde{\Omega}_{h}}, \ \ U_1^n \approx u_1(r,\theta;t_n)|_{\Tilde{\Omega}_{h}}, \ \ U_2^n \approx u_2(r,\theta;t_n)|_{\Tilde{\Omega}_{h}}.
\end{align*}
Then, the second order fully-discrete schemes for the pACOK equation follows: given initial condition $U^{-1} = U^{0} = u_0(r,\theta)|_{\Tilde{\Omega}_{h,\text{disk}}}$, for $n\in[\mathbb{N^+}]$, find $U^{n+1} = (U^{n+1})_{ij} \in \mathcal{M}_h$ such that   
\begin{align}\label{eqn:BDF_OK_fully}
    \frac{3U^{n+1}-4U^{n}+U^{n-1}}{2\tau}
    = \ & \epsilon\Delta_h U^{n+1} - \frac{1}{\epsilon}\left[2W'(U^n)-W'(U^{n-1})\right] \nonumber\\
    &  -\frac{\kappa_h}{\epsilon}\left(U^{n+1}-2U^{n}+U^{n-1}\right) -\gamma\beta_h(-\Delta)_h^{-1}\left(U^{n+1}-2U^{n}+U^{n-1}\right) \nonumber\\
    & -\gamma(-\Delta)_h^{-1}\left(2U^n-U^{n-1}-\omega\right)  -M\left[\langle2U^n-U^{n-1},1\rangle_h - \omega|\Tilde{\Omega}_h|\right]
\end{align}
where $\kappa_h,\beta_h\geq0$ are stabilization constants, the stabilizer $\frac{\kappa_h}{\epsilon}\left(U^{n+1}-2U^{n}+U^{n-1}\right)$ controls the growth of $W'$, and the stabilizer $\gamma \beta_h(-\Delta)_h^{-1}\left(U^{n+1}-2U^{n}+U^{n-1}\right)$ dominates the behavior of $(-\Delta)_h^{-1}$. 

Similarly, the second order fully-discrete scheme for the pACON equations reads: given initial condition $U_i^{-1} = U_i^{0} = u_{0,i}(r,\theta)|_{\Tilde{\Omega}_{h,\text{disk}}}$, $i = 1,2$, with stabilizer $\left(U^{n+1}-2U^{n}+U^{n-1}\right)$, the numerical solutions $(U_1^{n+1},U_{2}^{n+1}) = \left((U_1^{n+1})_{ij},(U_2^{n+1})_{ij}\right) \in \mathcal{M}_h\times\mathcal{M}_h$ for each $n\in[\mathbb{N^+}]$ are given by 
\begin{align}\label{eqn:BDF_NO_fully}
        &\frac{3U^{n+1}_i-4U^{n}_i+U^{n-1}_i}{2\tau} \nonumber\\ 
    = & \epsilon\Delta_h U^{n+1}_i + \frac{\epsilon}{2}\left(2\Delta_h U^{n+\text{mod}(i+1,2)}_{j}-\Delta_h U^{n-1+2\cdot\text{mod}(i+1,2)}_{j}\right) \nonumber \\
    & - \frac{1}{\epsilon}\left[2\frac{\partial W_2}{\partial u_{i}}(U^{n+\text{mod}(i+1,2)}_1,U^{n}_{2})-\frac{\partial W_2}{\partial u_{i}}(U^{n-1+2\cdot\text{mod}(i+1,2)}_1,U^{n-1}_{2})\right] \nonumber \\
    &  -\frac{\kappa_{i,h}}{\epsilon}\left(U^{n+1}_i-2U^{n}_i+U^{n-1}_i\right)  - \gamma_{ii}\beta_{i,h}(-\Delta)_{h}^{-1}\left(U^{n+1}_i-2U^{n}_i+U^{n-1}_i\right)  \nonumber\\
    & -\gamma_{ii}(-\Delta)_{h}^{-1}\left[2U^n_i-U^{n-1}_i-\omega_i\right]  -\gamma_{ij}(-\Delta)_{h}^{-1}\left[2U^{n+\text{mod}(i+1,2)}_j-U^{n-1+2\cdot\text{mod}(i+1,2)}_j-\omega_j\right]\nonumber\\
    & -M_i\left[\langle2U_i^n-U_i^{n-1},1\rangle_h - \omega_i|\Tilde{\Omega}_{h}|\right],
\end{align}
for $i = 1,2$ and $j\neq i$, and stabilization constants $\kappa_{1,h},\kappa_{2,h}, \beta_{1,h},\beta_{2,h}\geq0$.

To consider the energy stability for the fully-discrete schemes for the pACOK (\ref{eqn:BDF_OK_fully}) and pACNO (\ref{eqn:BDF_NO_fully}) equations, we define the discrete energy functional for the OK model 
\begin{align}\label{eqn:pOK_discrete}
E^{\text{pOK}}_{h}(U^{n}) =  \frac{\epsilon}{2}\|\nabla_{h}U^{n}\|_{L^2,h}+\frac{1}{\epsilon}\left\langle W(U^{n}),1 \right\rangle_h + \frac{\gamma}{2}\|(-\Delta)_{h}^{-\frac{1}{2}}\left(U^{n}-\omega\right)\|_{L^2,h} + \frac{M}{2} \left(\left\langle U^{n}, 1 \right\rangle_h - \omega|\Omega|\right),
\end{align}
and the discrete energy functional for the NO model
 \begin{align}\label{eqn:pNO_discrete}
E^{\text{pOK}}_{h}(U_1^{n},U_2^{n}) =  & \frac{\epsilon}{2}\left(\|\nabla_{h}U_1^{n}\|_{L^2,h}+\|\nabla_{h}U_2^{n}\|_{L^2,h}+\left\langle \nabla_{h}U_1^{n},\nabla_{h}U_2^{n}\right\rangle\right)+\frac{1}{\epsilon}\left\langle W_2(U_1^{n},U_2^{n}),1 \right\rangle_h \nonumber\\
& + \sum\limits_{i,j=1}^{2}\frac{\gamma_{ij}}{2}\left\langle(-\Delta)_{h}^{-1}\left(U_i^{n}-\omega\right),\left(U_j^{n}-\omega\right)\right\rangle_{h} + \sum\limits_{i=1}^{2}\frac{M_i}{2} \left(\left\langle U_i^{n}, 1 \right\rangle_h - \omega_i|\Omega|\right),
\end{align}
Then we establish the following energy stability for the fully discrete schemes (\ref{eqn:BDF_OK_fully}, \ref{eqn:BDF_NO_fully}). The proofs of the theorems are similar to that of Theorem \ref{thm:OK/NO_energy}, the only difference is to substitution of $(-\Delta)^{-1}$ with $(-\Delta)^{-1}_{h}$. We omit the details for brevity.

\begin{theorem}\label{thm:OK/NO_energy_fully} 
For the fully discrete scheme (\ref{eqn:BDF_OK_fully}) of pACOK equation, we define a modifying energy functional as
 \begin{align*}
\tilde{E}_h^{\text{pOK}}(U^{n},U^{n-1} ) = E_h^{\text{pOK}}(U^{n})+\left(\frac{\kappa_h}{2\epsilon}+\frac{1}{4\tau}+C_h\right)\|U^{n}-U^{n-1}\|_{L^2,h}^2+\frac{\gamma \beta_h}{2}\|(-\Delta)^{-\frac{1}{2}}\left(U^{n}-U^{n-1}\right)\|_{L^2,h}^2.
\end{align*}
where $\{U^{n}\}_{n=1}^N$ is generated by scheme (\ref{eqn:BDF_OK_fully}). If $\frac{\kappa_h}{2\epsilon}+\frac{1}{4\tau}\ge0$, $\beta_h\ge0$, and $\frac{1}{\tau}\ge 3 C_h$, then
\begin{align*}
    \tilde{E}^{\text{pOK}}_h (U^{n+1},U^{n} ) \le \tilde{E}^{\text{pOK}}_h (U^{n},U^{n-1} ),
\end{align*}
where $C_h$ is a generic constant independent of the time step size $\tau$.

For the pACNO system (\ref{eqn:BDF_NO_fully}), the modifying energy functional is defined by
\begin{align*}
    &\tilde{E}_h^{\text{pNO}}[U_1^{n},U_1^{n-1},U_2^{n},U_2^{n-1}] \\
    =\ & E_h^{\text{pNO}}[U_1^{n},U_2^{n}]+\sum_{i=1}^{2}\left[\left(\frac{\kappa_{i,h}}{2\epsilon}+\frac{1}{4\tau}+C_{i,h}\right)\|U_i^{n}-U_i^{n-1}\|_{L^2,h}^2+\frac{\gamma_{ii} \beta_{i,h}}{2}\|(-\Delta_h)^{-\frac{1}{2}}(U_i^{n}-U_i^{n-1})\|^2_{L^2,h}\right],
\end{align*}
where $\{U_1^{n},U_2^{n}\}_{n=1}^N$ is generated by scheme (\ref{eqn:BDF_NO_fully}). If $\kappa_{i,h},\beta_{i,h}\geq0$, $i=1,2$, and $\tau\leq \min\{\frac{1}{3C_{1,h}},\frac{1}{3C_{2,h}}\}$, then
\begin{align*}
     \tilde{E}^{\text{pNO}}[U_1^{n+1},U_1^{n},U_2^{n+1},U_2^{n}]\leq  \tilde{E}^{\text{pNO}}[U_1^{n},U_1^{n-1},U_2^{n},U_2^{n-1}],
\end{align*}
where $C_{i,h}$, $i = 1,2$ are generic constants independent of the time step size $\tau$.
\end{theorem}

To implement the fully-discrete scheme (\ref{eqn:BDF_OK_fully}), we take $\beta_h = 0$ and repeat the following algorithm: for $n = 1, 2, \cdots$,
\begin{enumerate}
\item Evaluate $(-\Delta_h)^{-1}(2U^{n}-U^{n-1}-\omega)$ using the Helmholtz solver with $\alpha = 0$ in section \ref{subsec:poisson};
\item Insert $(-\Delta_h)^{-1}(2U^{n}-U^{n-1}-\omega)$ into the right hand side of (\ref{eqn:BDF_OK_fully}) and solve another Helmholtz equation
    \begin{align}\label{eqn:OK_fully_Scheme}
    \left(-\epsilon\Delta_h+\frac{3}{2\tau}+\frac{\kappa_h}{\epsilon}\right)U^{n+1} = F_h^{n,n-1},
    \end{align}
where the term $F_h^{n,n-1}$ can be explicitly represented by
    \begin{align}\label{eqn:OK_fully_Scheme_RHS}
        F_h^{n,n-1} = & \frac{2}{\tau}U^n-\frac{1}{2\tau}U^{n-1} - \frac{1}{\epsilon}\left[2W'(U^n)-W'(U^{n-1})\right]+\frac{\kappa_h}{\epsilon}\left(2U^{n}-U^{n-1}\right) \nonumber\\
        & -\gamma(-\Delta)_h^{-1}\left(2U^n-U^{n-1}-\omega\right)  -M\left[\langle2U^n-U^{n-1},1\rangle_h - \omega|\Tilde{\Omega}_h|\right].
    \end{align}
\end{enumerate}

Similar algorithm applies to solving the fully-discrete scheme (\ref{eqn:BDF_NO_fully}) for the pACNO system with $\beta_{i,h} = 0$.

\section{Numerical Experiments}\label{sec:numerical}
In this section, we present some numerical experiments for the pACOK and pACNO equations in the disk domain $\Omega = [0,1]\times [0,2\pi) \subset \mathbb{R}^2$. By imposing homogeneous Neumann boundary conditions in 
radial ($r$) direction and periodic boundary conditions in angular ($\theta$) direction, we adopt the ultraspectral spectral method in the spatial domain and BDF schemes (\ref{eqn:BDF_OK_fully}) and (\ref{eqn:BDF_NO_fully}) with proper stabilizers to explore the coarsening dynamics and the equilibrium states of OK and NO models. For the disk domain, we take the uniform mesh grid with $N_{\theta} = 2^9$ in angular direction and Chebyshev mesh grid with $N_{r} = 2^9+1$ in radial direction. Denote the mesh size $h = \frac{2\pi}{N_{\theta}}$. The stopping criterias for the time iteration are set to 
\begin{align*}
    \frac{\|U^{n+1}-U^{n}\|_{h,L^{\infty}}}{\tau} \leq 10^{-5}
\end{align*}
for pACOK equation with approximate solution $U^{n}$ at $n$-th step, and 
\begin{align*}
     \frac{\|U_1^{n+1}-U_1^{n}\|_{h,L^{\infty}}}{\tau}+\frac{\|U_2^{n+1}-U_2^{n}\|_{h,L^{\infty}}}{\tau} \leq 10^{-5}
\end{align*}
for pACNO equation with approximate solution $U_1^{n},U_2^{n}$ at $n$-th step.

\subsection{Rate of convergence}
In this subsection, we test the rate of convergence of the second-order semi-implicit scheme for both pACOK and pACNO equations (\ref{eqn:pACOK},\ref{eqn:pACNO}) in the disk domain. 

For the pACOK equation, we choose the following initial condition:
\begin{align*}
    U^0(x,y) = \left\{
    \begin{array}{ll}
       1  &  \text{if } x^2 + (y-0.2)^2 < r_0^2,\\
       0  &  \text{otherwise},
    \end{array}
    \right.
\end{align*}
where $r_0 = \sqrt{\omega}+0.1$. While using the parameter $\omega = 0.15$, $\gamma = 100$, $\kappa = 1000$, $\beta = 5$ and $M = 1000$, we take the fully discrete scheme (\ref{eqn:pACOK_scheme}) for the numerical experiments. The test is conducted for $\epsilon$ values of $25h$ and $20h$, up to time $T=0.01$. The benchmark solution used for comparison has a time step size of $\tau = 1\times 10^{-6}$.

While utilizing the second-order scheme (\ref{eqn:pACOK_scheme}) to solve the pACOK equation for the binary system, Table \ref{table:pACOK_rateofconv} presents the error and corresponding rate of convergence at a fixed time $T=0.01$. It can be observed that the numerical rates for each $\epsilon = 25h, \ 20h$ are approximately equal to $2$ for small time step size $\tau$, which is  in accordance with the theoretical prediction.

\begin{table}[h!]
\begin{center}
\begin{tabular}{ |c||c|c|c|c|  }
\hline
 & \multicolumn{2}{|c|}{$\epsilon = 25h$} & \multicolumn{2}{|c|}{$\epsilon = 20h$} \\
\hline
$\tau$ & Error & Rate & Error & Rate \\
\hline
5.000e-4 & 1.37132e-1 & -       & 1.69732e-1 &  - \\
2.500e-4 & 4.09459e-2 & 1.74377 & 5.86987e-2 & 1.53186\\
1.250e-4 & 1.13057e-2 & 1.85667 & 1.62430e-2 & 1.85351\\
6.250e-5 & 3.27654e-3 & 1.78680 & 4.64194e-3 & 1.80702\\
3.125e-5 & 9.58757e-4 & 1.77294 & 1.33686e-3 & 1.79587\\
1e-6 (benchmark) & - &- &-&- \\
\hline
\end{tabular}
\end{center}
\caption{Rate of convergence for the pACOK equation in the binary system with parameters $\omega = 0.15$, $\gamma = 100$, $\kappa = 1000$, $\beta = 5$ and $M = 1000$.}
\label{table:pACOK_rateofconv}
\end{table}

For the pACNO equations, we use two separate disks as the initial data:
\begin{align*}
        U_1^0(x,y) = \left\{
    \begin{array}{ll}
       1  &  \text{if } (x-0.4)^2 + (y+0.3)^2 < r_1^2,\\
       0  &  \text{otherwise},
    \end{array}
    \right. \\
        U_2^0(x,y) = \left\{
    \begin{array}{ll}
       1  &  \text{if } (x+0.4)^2 + (y-0.3)^2 < r_2^2,\\
       0  &  \text{otherwise},
    \end{array}
    \right.
\end{align*}
where $r_1 = \sqrt{\omega_1}+0.05$ and $r_2 = \sqrt{\omega_2}+0.05$. In the numerical simulation, we take the parameters $\omega_{1} = \omega_{2} = 0.09$, $\gamma_{11} = \gamma_{22} = 500$, $\gamma_{12} = \gamma_{21} = 0$, $\kappa_1 = \kappa_2 = 1000$, $\beta_1 = \beta_2 = 0$ and $M_1 = M_2 = 1000$ and use the fully discrete scheme (\ref{eqn:pACNO_scheme}). The test is also performed up to $T = 0.01$ for $\epsilon = 25h$ and $20h$, and the benchmark solution is calculated with the time step size of $\tau = 1\times 10^{-6}$.

Table \ref{table:pACNO_rateofconv} shows the errors and the convergence rates for the pACNO equations in the ternary system solved using the second-order scheme (\ref{eqn:pACNO_scheme}) at a fixed time $T=0.01$. Similar to the pACOK equation, the numerical rates closely match the theoretical value of $2$ for each $\epsilon$ value of $25h$ and $20h$, as well as for a small time step size $\tau$.

\begin{table}[h!]
\begin{center}
\begin{tabular}{ |c||c|c|c|c|  }
\hline
 & \multicolumn{2}{|c|}{$\epsilon = 25h$} & \multicolumn{2}{|c|}{$\epsilon = 20h$} \\
\hline
$\tau$ & Error & Rate & Error & Rate \\
\hline
5.000e-4 & 1.31507e-1 & -       & 1.59813e-1 &  - \\
2.500e-4 & 3.88543e-2 & 1.75900 & 5.23678e-2 & 1.60963\\
1.250e-4 & 1.06830e-2 & 1.86276 & 1.45313e-2 & 1.84952\\
6.250e-5 & 3.18168e-3 & 1.74745 & 4.25723e-3 & 1.77117\\
3.125e-5 & 1.04297e-3 & 1.60908 & 1.35478e-3 & 1.65185\\
1e-6 (benchmark) & - &- &-&- \\
\hline
\end{tabular}
\end{center}
\caption{Rate of convergence for the pACNO equation in the ternary system with parameters $\omega_{1} = \omega_{2} = 0.09$, $\gamma_{11} = \gamma_{22} = 500$, $\gamma_{12} = \gamma_{21} = 0$, $\kappa_1 = \kappa_2 = 1000$, $\beta_1 = \beta_2 = 0$ and $M_1 = M_2 = 1000$.}
\label{table:pACNO_rateofconv}
\end{table}

\subsection{Coarsening Dynamics of OK model}
In this subsection, we focus on the dynamics of the pACOK equation with the scheme (\ref{eqn:BDF_OK_fully}). When the relative volume $\omega$ of one phase (species $A$) is small compared to the total volume, and the strength $\gamma$ of long-range interactions is relatively large, i.e. $\omega\ll 1$ and $\gamma\gg1$, the pACOK equation results in an equilibrium of bubble assembly in which one phase (species $A$) is embedded into the other one (species $B$). 

The initial data is randomly given by the uniform mesh of the disk domain $\Omega = [0,1]\times [0,2\pi)$ with the mesh size $32h$ in $r$-direction and $8h$ in $\theta$-direction. In MATLAB, the random initials are easily generated by the command 
\begin{align*}
    \left[\begin{array}{cc}
        \texttt{repelem}(\texttt{rand}((N_r-1)/\texttt{ratio},N_{th}/(\texttt{ratio}/4)),\texttt{ratio},\texttt{ratio}/4)^T \\
        \texttt{repelem}(\texttt{rand}(N_{th}/(\texttt{ratio},1)),\texttt{ratio},1)^T
    \end{array}\right]^{T}
\end{align*}
with {\texttt{ratio = 32}} in the numerical experiments for pACOK equation. At last, the parameter values are chosen as $\tau = 5\times10^{-4}$, $\omega = 0.15$, $\kappa = 2000$, $\beta = 0$ and $M = 2000$.

The coarsening dynamics of the pACOK equation and the equilibria in the binary system are shown in Figure \ref{fig:binary_system} with $\gamma = 2500,\ 3500$, and $4000$ respectively. For each subfigure, we insert the snapshots taken at four different times $t$, with colored titles corresponding to the colored marker on the monotone-decreasing energy curve. With the value of $\gamma$ increasing and other parameters fixed, we observe that the stronger long-range repulsive force leads to an increase in the number of bubbles. The results show in Figure \ref{fig:binary_system} (top) with $\gamma = 2500$, Figure \ref{fig:binary_system} (middle) with $\gamma = 3500$ and Figure \ref{fig:binary_system} (bottom) with $\gamma = 4000$. In Figure \ref{fig:binary_system} (top), starting from the random initial, the phase separation arises in a very short time period and a group of bubbles with different sizes appear as shown in $t=0.5$, then the small bubbles disappear and other bubbles in the interior of the disk grow into the same size which is shown in $t=3,10,50$, similar to those in the square domain \cite{Choi_Zhao2021}. Interestingly, different from the square domain, we can observe half bubbles appearing as an interaction with the boundary of the disk, which can also be found from Ren and Shoup's recent work \cite{Ren_Shoup2020}.

\begin{figure}[h]
    \begin{center}
    \includegraphics[width=1\textwidth]{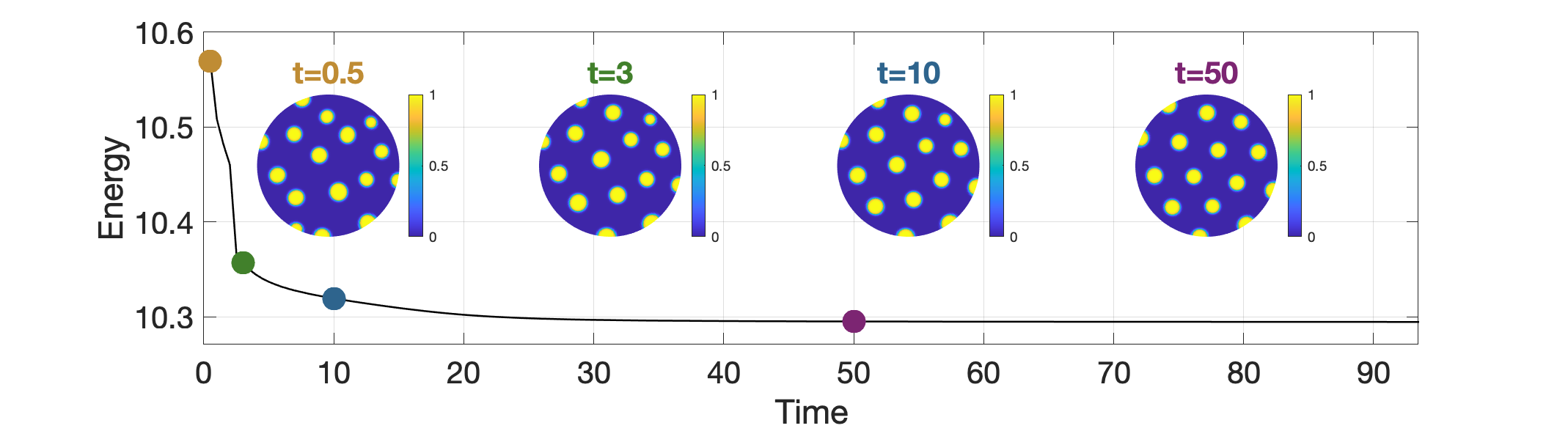}
    \includegraphics[width=1\textwidth]{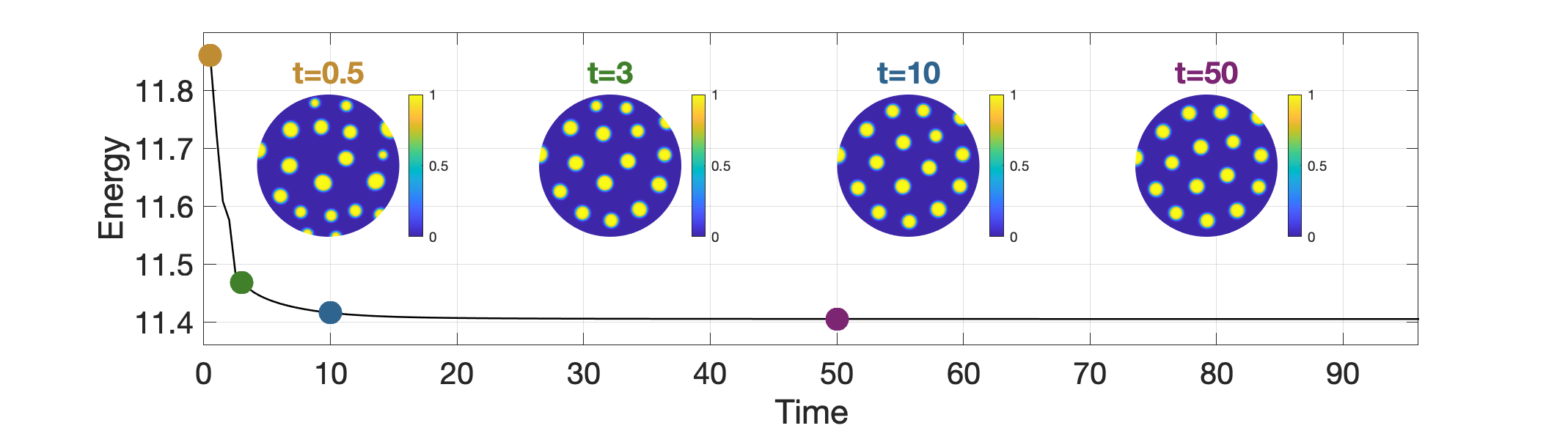}
    \includegraphics[width=1\textwidth]{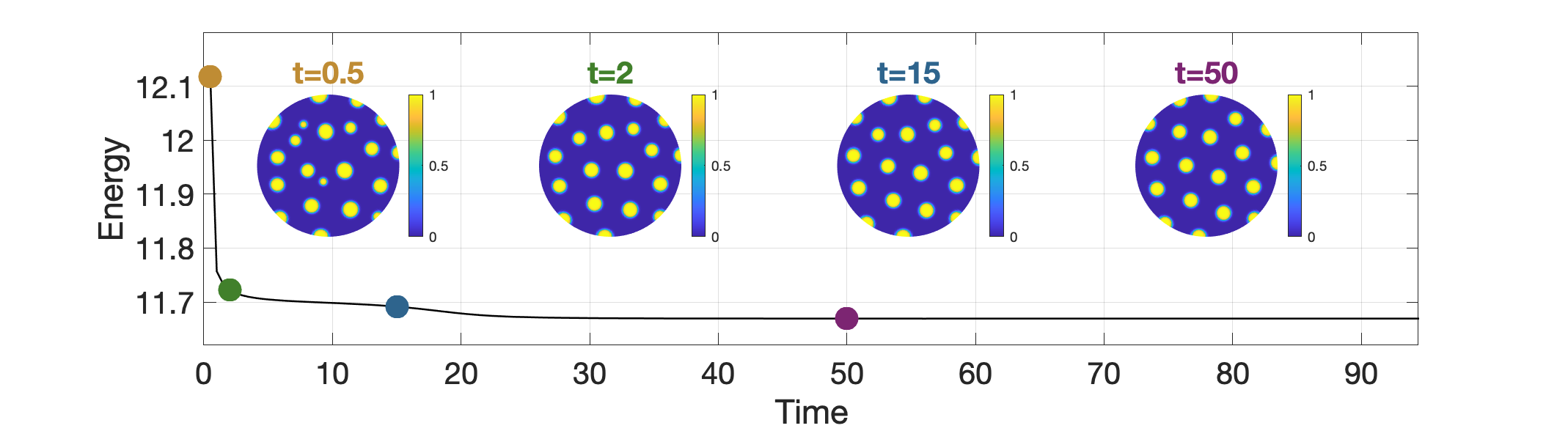}
    \end{center}
    \caption{Coarsening dynamics in binary system on a unit disk with $\gamma = 2500$(top), $\gamma = 3500$(middle), and $\gamma = 4000$(bottom).}
    \label{fig:binary_system}
\end{figure}

\begin{figure}[h]
    \begin{center}
    \includegraphics[width=1\textwidth]{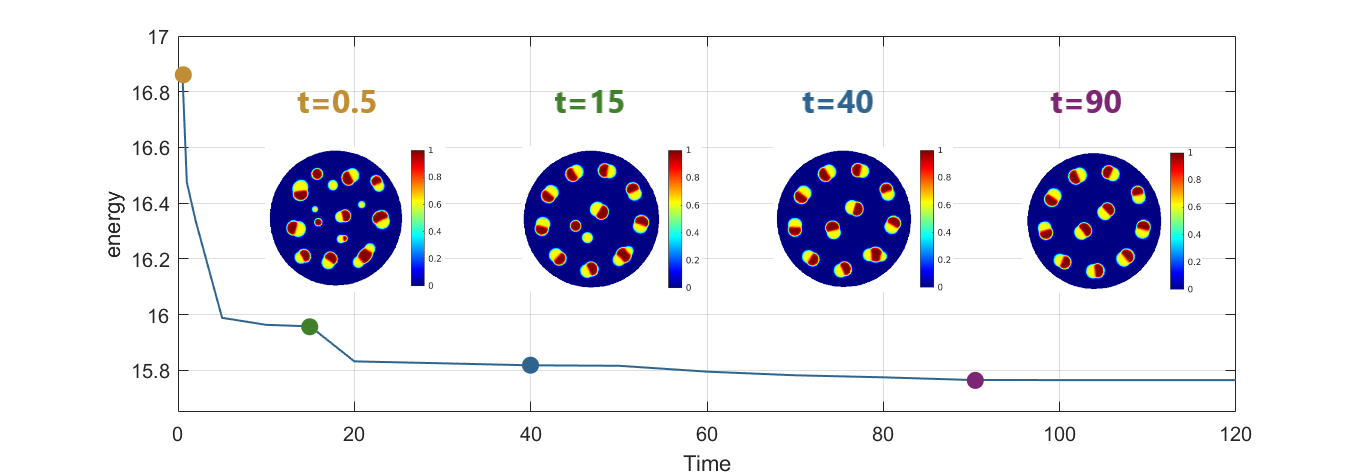}
    \includegraphics[width=1\textwidth]{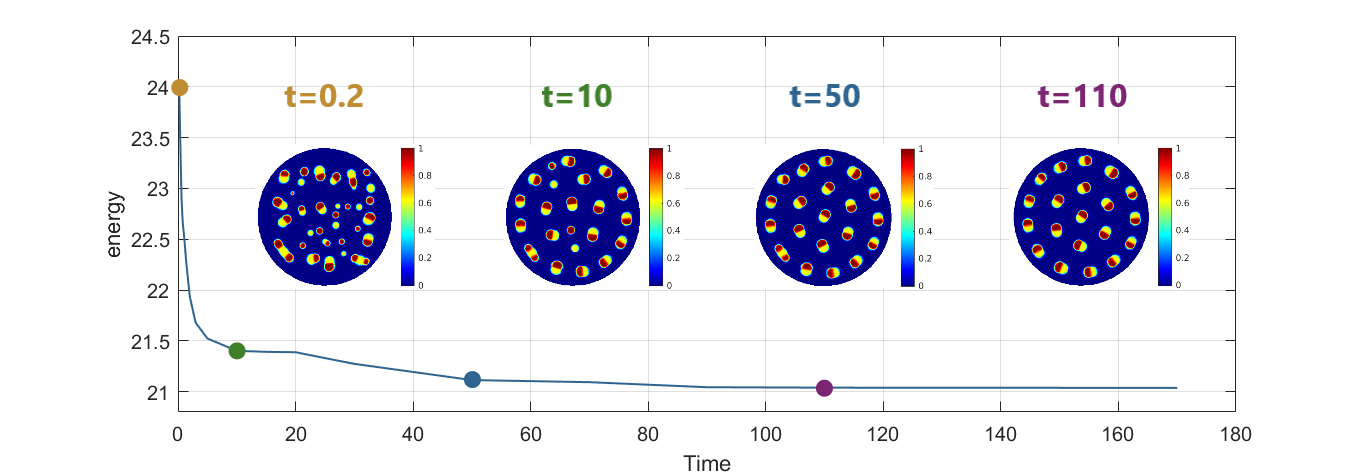}
    \end{center}
    \caption{Coarsening dynamics in ternary system on a unit disk with $\gamma_{11} =  \gamma_{22}=6000, \gamma_{12} = \gamma_{21}=0$ (top), and $\gamma_{11} = \gamma_{22} = 15000,\gamma_{12} = \gamma_{21}=0$ (bottom).}
    \label{fig:gamma11_effect}
\end{figure}
    
\begin{figure}[p]
    \includegraphics[width=1\textwidth]{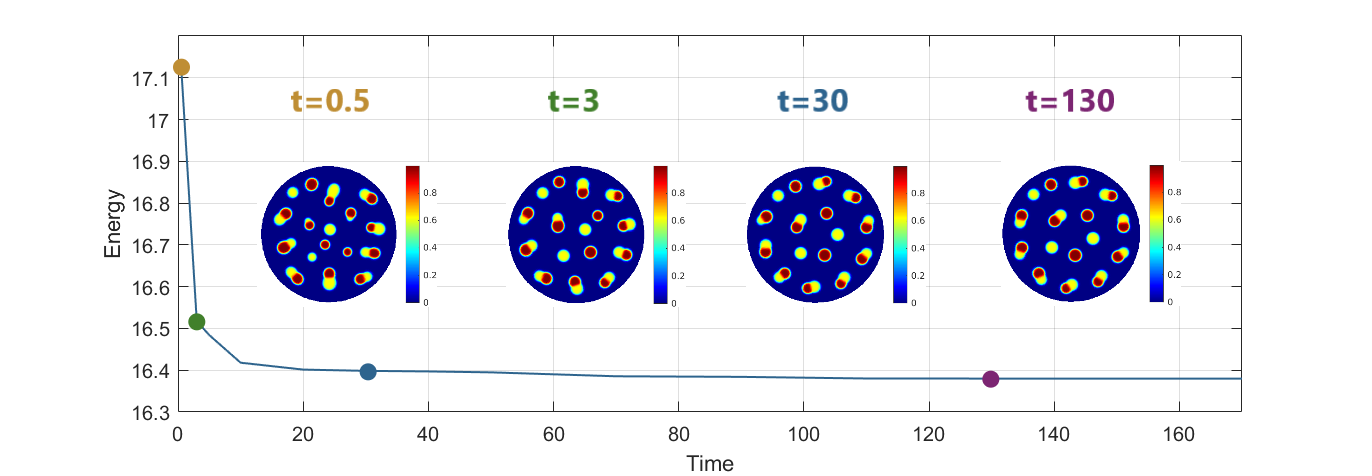}
    \includegraphics[width=1\textwidth]{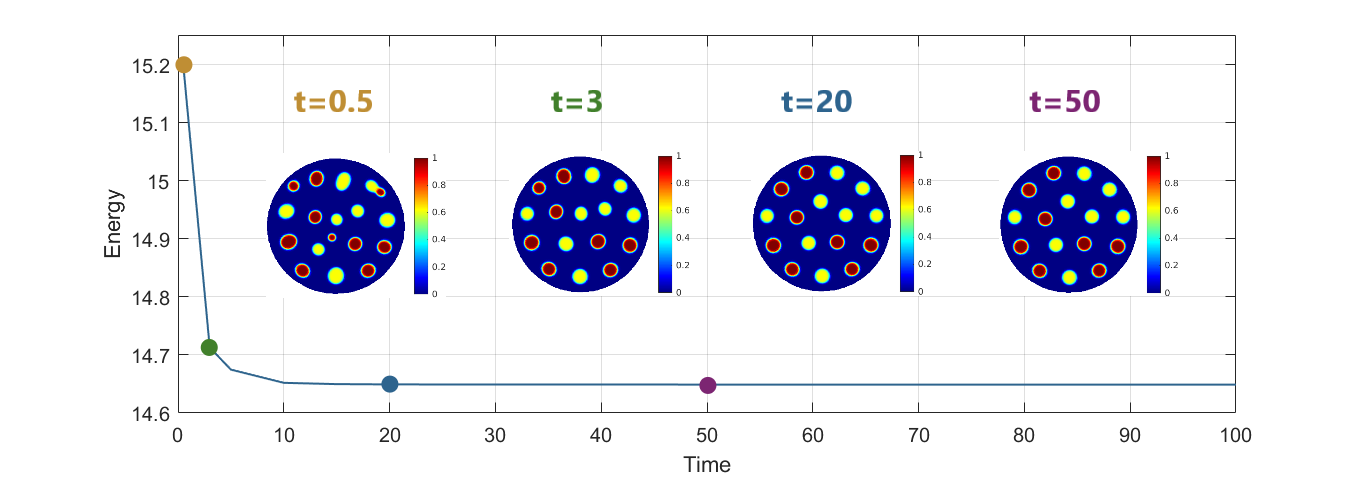}
    \includegraphics[width=1\textwidth]{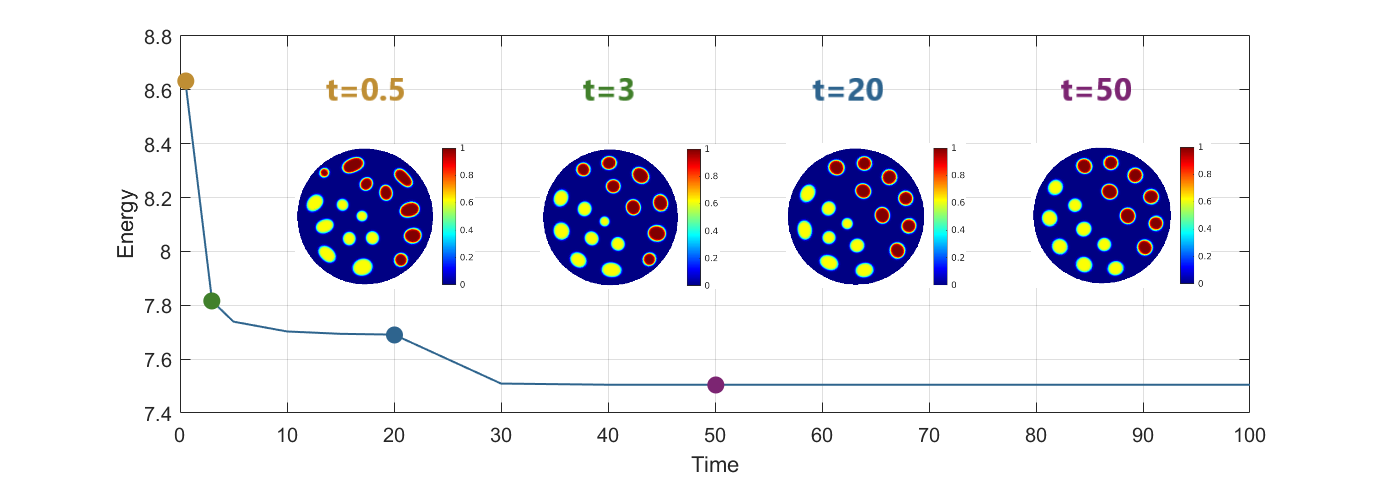}
    \centering
    \caption{Coarsening dynamics in ternary system on a unit disk with $\gamma_{11} =  \gamma_{22}=6000, \gamma_{12} = \gamma_{21}=3000$ (top), $\gamma_{11} = \gamma_{22} = 6000,\gamma_{12} = \gamma_{21}=6000$ (middle), and $\gamma_{11} = \gamma_{22} = 6000, \gamma_{12} = \gamma_{21}=8000$ (bottom).}
    \label{fig:gamma12_effect}
\end{figure}

\begin{figure}[h]
    \begin{center}
    \includegraphics[width=1\textwidth]{gamma11_6000_02.png}
    \includegraphics[width=1\textwidth]{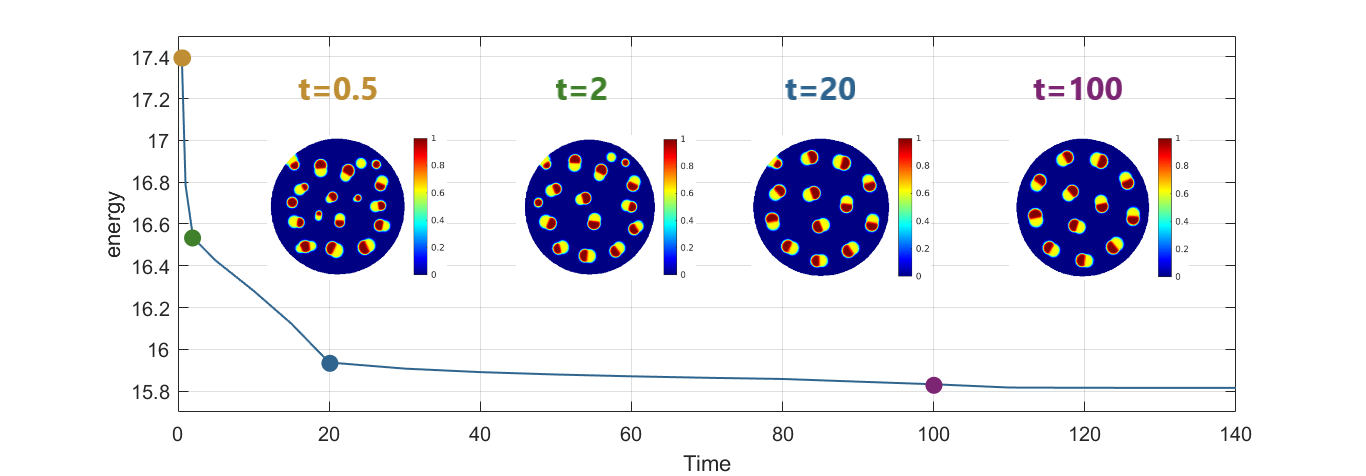}
    \end{center}
    \caption{different equilibrium in the ternary system on a unit disk with $\gamma_{11} =  \gamma_{22}=6000, \gamma_{12} = \gamma_{21}=0$, and the equilibrium energy: $E_{\text{(a)}}<E_{\text{(b)}}$.}
    \label{fig:energy_compare}
\end{figure}

\subsection{Coarsening Dynamics of NO model}\label{subsec:NO}
In this subsection, we numerically study the dynamics and equilibrium of the pACNO equations. While using the suggested scheme (\ref{eqn:BDF_NO_fully}) and taking $\omega\ll 1$ and relatively large $\gamma$, we can observe the dynamics and equilibria as several types of double bubble assemblies by choosing different long-range strengths $(\gamma_{ij})_{i,j=1,2}$ in pACNO equations. In our numerical experiments, we take the long-range repulsive force $\gamma_{11} = \gamma_{22}$ and $\gamma_{12} = \gamma_{21}$.

While choosing the random initial similar to the binary system, the phase separation of the ternary system costs several days in the experiments. Therefore, to avoid the long time consumption and study the equilibrium of the pattern formation on the interior of the domain, we construct the initial configuration in the ternary system in the disk domain by random circles on the interior of the disk with random centers and radius, each random circle is allowed overlapping with others. We call it a 'semi-random initial' in our following work. Moreover, we fix time step $\tau = 5\times10^{-4}$, relative area $\omega_1=\omega_2 = 0.09$, stabilized constants $\kappa_1 = \kappa_2 = 2000$, $\beta_1=\beta_2=0$ and $M_1 =M_2= 2000$.

\subsubsection{Coarsening dynamics and pattern formation in the disk domain}
The coarsening dynamics of the ternary system in the disk domain are shown in Figure \ref{fig:gamma11_effect}. Similar to the binary system, the insets are snapshots taken at four different times $t$. In Figure \ref{fig:gamma11_effect} (top), we take $\gamma_{11}=\gamma_{22} = 6000$, $\gamma_{12}=\gamma_{21} = 0$. Starting from the initial generated by random circles, in a very short period of time, the phase separation takes place, leading to the coexistence of double-bubbles and single-bubble patterns (t=0.5). Next, some tiny single bubbles disappear (t=15) and others begin merging into the double-bubbles pattern (t = 40). Finally, a pattern of all double bubbles is formed and composed of two circular layers, in which two double bubbles stay in the middle layer, surrounded by a layer of nine double bubbles, each with the head pointing to the tail of another, forming a circle, which we call the 'head-to-tail' pattern (t = 90).

While increasing the repulsive force $\gamma_{11}=\gamma_{22} = 15000$, and fixing other parameters, the repulsive interaction between the same component becomes stronger and leads to more double-bubbles as shown in Figure \ref{fig:gamma11_effect} (bottom). Moreover, the larger repulsive force induces three circular layers, in which one double bubble lies within the inner layer (the center of the disk domain), surround by six double bubbles forming a 'head-to-tail' circle in the middle layer, and the outside layer contains thirteen double bubbles with the 'head-to-tail' circle.

This is the most important numerical finding from our model system. To our best knowledge, it is the first time that the 'head-to-tail' pattern is numerically discovered in the study of the pACNO equations in the disk domain.

\subsubsection{The effect of $\gamma_{12}$}
The effect of $\gamma_{12} = \gamma_{21}$ can be found in Figure (\ref{fig:gamma11_effect}) and (\ref{fig:gamma12_effect}). While the long-range force $\gamma_{11}$ ($\gamma_{22}$, respectively) is dedicated to splitting the same component (red and yellow respectively), the repulsive strength $\gamma_{12} = \gamma_{21}$ can be viewed as the splitting force between red and yellow components, namely, for fixed $\gamma_{11}=\gamma_{22}$, the increase of $\gamma_{12}$ tends to separate the red and yellow constituents.

While taking $\gamma_{11}=\gamma_{22} = 6000$, we increase the value of $\gamma_{12}$ from $0$ (Figure \ref{fig:gamma11_effect} (top)) to $3000$ (Figure \ref{fig:gamma12_effect} (top)), $6000$ (Figure \ref{fig:gamma12_effect} (middle)) and $8000$ (Figure \ref{fig:gamma12_effect} (bottom)). By taking $\gamma_{12} = \gamma_{21} = 0$ in Figure \ref{fig:gamma12_effect} (top), there is no strength to separate the red and yellow components, thus they adhere together which results in all-double-bubble assembly. While increasing $\gamma_{12} = \gamma_{21} = 3000$, some of the double bubbles separate into single red and yellow bubbles, which lead to the coexistence of double-bubbles and single-bubble patterns as shown in Figure \ref{fig:gamma12_effect} (top). In Figure \ref{fig:gamma12_effect} (middle), the strengths $\gamma_{12} = \gamma_{21} = 6000$ become larger and break all double bubbles, therefore the bubble assembly tends to be purely single bubbles. Figure \ref{fig:gamma12_effect} (bottom) shows that with even larger repulsive strength $\gamma_{12} = \gamma_{21} = 8000$, the red bubbles are completely separated away from yellow bubbles. Although there are no theoretical studies of the NO model in the disk domain, our numerical results are consistent with the experiments in the square domain from Wang, Ren, and Zhao's work \cite{Wang_Ren_Zhao2019}.

\subsubsection{Equilibrium of the ternary system: 'Head-to-tail' pattern}

Due to the nonconvexity of the energy functional of the NO model, pACNO system may display multiple equilibra with the same parameter values. For the experiment as shown in Figure \ref{fig:energy_compare}, we take $\gamma_{11} = \gamma_{22} = 6000$ and $\gamma_{12} = \gamma_{21} = 0$ for each subfigure. Starting from several different semi-random initials, we observe two equilibria as shown in Figure \ref{fig:energy_compare}. In Figure \ref{fig:energy_compare} (top), the equilibrium has two 'head-to-tail' circles, with two double bubbles in the inside circle, and nine double bubbles in the outside circle and its energy is approximately $E_{(a)} = 15.7645$. With another semi-random initial, we can also observe two 'head-to-tail' circles as the equilibrium in Figure \ref{fig:energy_compare} (bottom), but different from the top one, the inside circle has three double bubbles and the outside circle has nine double bubbles, and the energy approximately is $E_{(b)} = 15.8156$. Comparing the equilibrium energy for each subfigure, it's easy to see that $E_{(a)} < E_{(b)}$, which indicates that the equilibrium on top is more energetically favorable than the one at bottom.

\section{Concluding remarks}

In this paper, we introduce a numerical method for studying the dynamics and equilibrium states of the OK and NO models in the disk domain. While applying the second-order BDF scheme in time and the ultraspherical spectral collocation method for spatial variables, we develop an energy stable scheme for the pACOK and pACNO equations in the disk domain with the Neumann boundary condition.

In the numerical experiments, we notice the phase separation and single bubble formation on both the interior and boundary of the disk by pACOK dynamics for the binary system when one species has a much smaller volume than the other. With the pACNO dynamics in the disk domain, the repulsive force between different species leads to various types of bubble assembly on the interior of the unit disk, the most important observation from the experiments is the all-double-bubbles pattern with the 'head-to-tail' circles for the ternary system on unit disk.

In the furture work, we will focus on a systematic study of the bubble assemblies for the binary system over disk domain. More specifically, we will explore all possible equilibria, each of which contains interior bubbles and boundary half-bubbles. This numerical investigation will be compared to the theoretical findings in Ren and Shoup's work \cite{Ren_Shoup2020}. For the ternary system, this work can be extended to the double-bubble patterns both in the interior and on the boundary of the disk. Another possible direction is to thoretically analyze the 'head-to-tail' pattern and study the parameter dependence of the patterns.

\section*{Acknowledgments} 

This work is supported by a grant from the Simons Foundation through Grant No. 357963 and NSF grant DMS-2142500.

%%%%%%%%%%%%%%%%%%%%%%%%%%%%%%%%%%%%%%%%%%%%%%%%%%%%%%%%%%%%%%%%%%%%%%%%%%%%%%%%%

%\section*{References}

\bibliography{OhtaKawasaki}

\end{document}